\documentclass[11pt]{amsart}
\usepackage[english]{babel}
\usepackage[margin=1.05in]{geometry}
\usepackage{amsmath,amssymb,amsthm,mathtools}
\usepackage{enumitem}
\usepackage{array}
\usepackage{booktabs}
\usepackage{microtype}
\usepackage{xcolor}
\usepackage{graphicx}
\usepackage{tikz}
\usepackage{pgfplots}
\pgfplotsset{compat=1.14}
\usetikzlibrary{arrows.meta,calc,positioning}
\usepackage[colorlinks=true,linkcolor=blue,citecolor=blue,urlcolor=blue]{hyperref}
 \usepackage[all]{xy}

\theoremstyle{plain}
\newtheorem{lem}{Lemma}[section]
\newtheorem{thm}[lem]{Theorem}
\newtheorem{cor}[lem]{Corollary}
\newtheorem{prop}[lem]{Proposition}

\theoremstyle{definition}
\newtheorem*{rem}{Remark}
\newtheorem{ex}[lem]{Example}

\numberwithin{equation}{section}
\newcommand{\NCF}[1]{[\![#1]\!]}

\newcommand{\R}{\mathbb{R}}
\newcommand{\Z}{\mathbb{Z}}
\newcommand{\C}{\mathbb{C}}

\newcommand{\Q}{\mathbb{Q}}
\newcommand{\K}{\mathbb{K}}
\newcommand{\RP}{{\mathbb{RP}}}
\newcommand{\pP}{{\mathbb{P}}}

\newcommand{\cM}{\mathcal{M}}
\newcommand{\Rc}{\mathcal{R}}

\newcommand{\Sc}{\mathcal{S}}

\newcommand{\id}{\textup{Id}}

\newcommand{\gr}{\textup{gr}}

\newcommand{\PSL}{\mathrm{PSL}}
\newcommand{\PGL}{\mathrm{PGL}}

\newcommand{\Bbar}{\overline{B}}
\newcommand{\crr}{\operatorname{cr}}

\title[$q$-cross-ratio]{The $q$-deformed cross-ratio: modular invariants
\\ and Coxeter friezes}

\author{Valentin Ovsienko}

\dedicatory{To Sergei Tabachnikov on his 70th birthday}

\date{}

\begin{document}

\begin{abstract}
We introduce and study a scalar $q$-deformation of the cross-ratio on $\mathbb P^1(\mathbb Q)$.
Our construction is based on the notion of $q$-deformed rational numbers
due to Morier-Genoud and the author.
The $q$-cross-ratio is invariant under $\mathrm{PSL}(2,\mathbb{Z})$,
while elements of determinant $-1$ of $\mathrm{PGL}(2,\mathbb{Z})$
act by $q\mapsto q^{-1}$.
A principal result is its relation to $q$-deformed Coxeter friezes
associated with rational polygons.
The expansion at $q=e^h$ yields an  
algebraically independent sequence of modular invariants and relative invariants, 
although this sequence does not separate modular orbits.
We compute the first two nonconstant coefficients of this expansion explicitly.
\end{abstract}

\maketitle

\thispagestyle{empty}

\tableofcontents

\section{Introduction}

The theory of $q$-deformed rational and real numbers introduced in~\cite{MGO1,MGOexp} 
has a projective-geometric origin: 
it is based on a $q$-deformation of the standard action of $\PSL(2,\Z)$ on the rational projective line.  
The construction extends to $\PGL(2,\Z)$; see~\cite{MGO1,LMGadv,Jou}.  
For a survey, see~\cite{MGOSur}.
The present paper brings these $q$-numbers back to projective geometry.

We study the geometry of the projective line over the field of rational numbers.
The situation changes drastically when the full projective symmetry
group $\PGL(2,\Q)$ is replaced by its arithmetic subgroup
$\PGL(2,\Z)$.  The full group $\PGL(2,\Q)$ acts sharply $3$-transitively on $\pP^1(\Q)$, 
so the ordinary cross-ratio classifies ordered quadruples modulo $\PGL(2,\Q)$.
By contrast, $\PGL(2,\Z)$ is not even $2$-transitive, and arithmetic invariants already occur for pairs. 
This integral structure is also the source of the $q$-deformation,
which connects this geometry with combinatorics and mathematical physics.
Thus the modular $q$-deformation and the
richer invariant theory have a common arithmetic origin.

The most important motivation for this work 
is related to the fact that the classical cross-ratio
is ubiquitous.
Passing to \(\PGL(2,\mathbb Z)\) creates a much richer orbit structure, 
and simultaneously opens the door to the modular \(q\)-deformation studied in the paper.
We hope that the $q$-deformed counterpart also has many interesting applications to
modular geometry, Coxeter friezes, dynamical systems, 
and arithmetic invariants of rational configurations.

The cross-ratio is the essence of projective geometry.
This is why the idea of its $q$-deformation is very tempting.
The $q$-cross-ratio provides a natural way to package infinitely
many algebraically independent modular invariants and relative
invariants.  It should not, however, be expected to restore the
classical one-parameter classification and separate rational quadruples.

\subsection*{The cross-ratio.}
The action of $\PGL(2,\mathbb K)$ on~$\pP^1$ is $3$-transitive, for any field $\mathbb K$.
Hence ordered triples of distinct points carry no nonconstant projective invariant.

The classical cross-ratio is an invariant of four pairwise distinct points
on the projective line~$\pP^1$ over an arbitrary field.
It can be defined explicitly by
\begin{equation}
\label{ClCR}
\crr(x_1,x_2;x_3,x_4)
=
\frac{(x_1-x_3)(x_2-x_4)}
     {(x_1-x_4)(x_2-x_3)},
\end{equation}
where $x_1,x_2,x_3$, and $x_4$ are affine coordinates of the points.
For the point at infinity we use the homogeneous expression below.
The cross-ratio is invariant under the action of $\PGL(2,\mathbb K)$
by fractional-linear transformations.
Write $x=(u:v)$ where $u,v$ are homogeneous coordinates,
and choose an arbitrary lift 
$$
X_i=
\begin{pmatrix}
u_i\\
v_i
\end{pmatrix}
$$
of each $x_i$ to the two-dimensional vector space.
For $i\neq j$, consider the Pl\"ucker determinants
$$
\Delta_{ij}
=
\det(X_i,X_j)
=
u_i v_j-u_j v_i.
$$
The cross-ratio can then be written as
\begin{equation}
\label{ClassEq}
\crr(x_1,x_2;x_3,x_4)
=
\frac{\Delta_{13}\Delta_{24}}
     {\Delta_{14}\Delta_{23}}.
\end{equation}
Formula~\eqref{ClCR} assumes finite affine coordinates, whereas~\eqref{ClassEq} also handles $\infty$.
For more information about the cross-ratio and classical references, see~\cite{Lab,OT}.

\subsection*{Normalized lift and Pl\"ucker invariants.}
For a rational point $x\in\pP^1(\Q)\simeq\Q\cup\{\infty\}$ 
we define the natural ``primitive'' integer lift $X\in\Z^2$.
If $x=\frac{a}{b}$, where $a,b\in\Z$, we will always assume $\gcd(a,b)=1$ and $b>0$.
Then the lift is
\begin{equation}
\label{LiftEq}
X=
\begin{pmatrix}
a\\ b
\end{pmatrix}.
\end{equation}
We also define $X_\infty=
\begin{pmatrix}
1\\ 0
\end{pmatrix}.
$

Given any family of ordered points $x_1,x_2,\ldots$,
consider the arithmetic Pl\"ucker determinants
\begin{equation}
\label{PluEq}
\Delta_{ij}
=
a_i b_j-a_j b_i,
\end{equation}
corresponding to the integer lifts of~$x_i,x_j$.
Their absolute values $|\Delta_{ij}|$
are $\PGL(2,\Z)$-invariants 
(whereas the initial Pl\"ucker determinants $\Delta_{ij}$ can change sign under the $\PGL(2,\Z)$-action).

The $q$-analogues of the Pl\"ucker determinants $\Delta_{ij}$ 
were studied in~\cite{MGO1}.
The $q$-deformation of a rational $\frac{r}{s}$ is a quotient of two coprime polynomials:
$$
\left[\frac{r}{s}\right]_{q}=\frac{\Rc(q)}{\Sc(q)}.
$$
Moreover, if $\frac{r}{s}>1$ then $\Rc$ and~$\Sc$ are monic polynomials with positive integer coefficients.
The polynomial
\begin{equation}
\label{qPluck}
\Delta^q_{ij}:=\Rc_i(q)\Sc_j(q)-\Rc_j(q)\Sc_i(q)
\end{equation}
was studied in~\cite{MGO1}.
In particular, if $\frac{r_i}{s_i}>\frac{r_j}{s_j}$, then $\Delta^q_{ij}$
has positive integer coefficients
(see~\cite{MGO1}, Theorem 2).
Moreover, the polynomials~\eqref{qPluck} are unimodal~\cite{OR}.
The $q$-deformed Pl\"ucker determinants satisfy the usual Plücker identity:
\begin{equation}
\label{PluckIdentq}
\Delta^q_{13}\Delta^q_{24}=
\Delta^q_{12}\Delta^q_{34}+\Delta^q_{14}\Delta^q_{23}.
\end{equation}

\subsection*{Introducing the $q$-cross-ratio.}
The main goal of this paper is to introduce and study a new scalar $q$-deformation of the cross-ratio.
The construction is based on the notion of a $q$-deformed rational number~\cite{MGO1}.

Consider the rational projective line $\pP^1(\Q)\simeq\Q\cup\{\infty\}$, where
$\infty$ is represented by~$\frac10$.
We use the quantization map
$$
[\,\cdot\,]_q:\pP^1(\Q)\longrightarrow\pP^1(\Q(q)),
\qquad [0]_q=0,
$$
which is equivariant for the $q$-deformed $\PSL(2,\Z)$-action.
The full $\PGL(2,\Z)$ action used in this article is semilinear: 
determinant-$(-1)$ elements replace $q$ by $q^{-1}$.
For finite $x$, the value $[x]_q$ is a rational function of $q$, 
represented as a quotient of polynomials
in $\Z[q]$.

The $q$-deformation of the cross-ratio will be defined by the following formula:
\begin{equation}
\label{Crq}
\crr_q(x_1,x_2;x_3,x_4)
:=
\crr\bigl([x_1]_q,[x_2]_q;[x_3]_q,[x_4]_q\bigr).
\end{equation}
The equivariance of $q$-rationals with respect to the $q$-deformed
modular action implies a natural covariance property under the simultaneous
classical $\PGL(2,\Z)$-action on the four rational arguments.
More precisely, the determinant $+1$ elements of $\PGL(2,\Z)$ preserve $\crr_q$, 
whereas determinant $-1$ elements send it to $\crr_{q^{-1}}$.

\begin{rem}
Throughout the paper, $[x]_q$ denotes the original, or right,
$q$-rational.  There is also a natural left $q$-rational, denoted by
$[x]_q^{\flat}$.  Replacing the four right $q$-rationals in~\eqref{Crq}
by left ones gives no independent invariant: the resulting cross-ratio is
$\crr_{q^{-1}}$.  More generally, one may choose independently the right
or left version in each of the four positions.  These mixed
$q$-cross-ratios are discussed in
Section~\ref{subsec:mixed-cross-ratios}.  They contain more information
than $\crr_q$, but even their complete collection does not separate
modular orbits.
\end{rem}

The determinant formula for the ordinary cross-ratio gives
\begin{equation}
\label{PluckCR}
\crr_q(x_1,x_2;x_3,x_4)=
\frac{\Delta^q_{13}\Delta^q_{24}}
     {\Delta^q_{14}\Delta^q_{23}},
\end{equation}
and we deduce the following.
If $x_1,x_2,x_3,x_4\in\pP^1(\Q)$ satisfy $x_1>x_2>x_3>x_4\geq0$, then 
$\crr_q(x_1,x_2;x_3,x_4)$ is a quotient of two polynomials in $\Z[q]$.
After reduction, however, positivity and unimodality of the numerator and denominator need not survive
(a simple example is $\crr_q(1/4,1/6;1/8,0)=1+q^2$).

It follows from~\eqref{PluckIdentq} that under the action of the symmetry group~$S_4$
the $q$-cross-ratio also transforms by the classical anharmonic transformations
$$
z,
\qquad
\frac{1}{z},
\qquad
1-z,
\qquad
\frac{1}{1-z},
\qquad
\frac{z}{z-1},
\qquad
\frac{z-1}{z}.
$$

\subsection*{Relation to $q$-deformed Coxeter frieze patterns.}
One remarkable combinatorial property of the classical cross-ratio
is its relation to Coxeter frieze patterns.
For details, see the survey by Morier-Genoud~\cite{MGsurvey}.
The notion of $q$-deformed Coxeter frieze patterns was introduced in~\cite{MGOqfrieze}.
The most important geometric property 
known to us is the relation between the $q$-deformed cross-ratio
and the $q$-deformed Coxeter frieze patterns.
More precisely, we consider the negative of an anharmonic transform 
$$
\gr_q = \frac{1}{\crr_q -1}.
$$

\begin{thm}
\label{CoxThmIntro}
For the $q$-deformation of an integral Conway--Coxeter frieze, 
with associated rational $n$-gon (with~$n\geq5$) as in Section~\ref{FrizSec}, 
each entry of the second nontrivial row is, 
up to an explicit monomial in $q$, 
the value of $\gr_q$ on four consecutive vertices.
\end{thm}

A more detailed statement is formulated in Theorem~\ref{thm:second-row}.
This theorem shows that the scalar deformation~\eqref{Crq} is compatible with the $q$-frieze geometry.

\subsection*{$\PGL(2,\Z)$-invariance, expansion around $q=1$.}
Since $\crr_1=\crr$, it is natural to consider the function $\crr_q$ in the vicinity of $q=1$.

Putting $q=e^h$, we can write the $q$-deformed cross-ratio as a 
Taylor expansion in the logarithmic coordinate $h=\log q$:
\begin{equation}
\label{CrT}
\crr_{e^h}
=
\sum_{n\geq0}\frac{\crr^{(n)}h^n}{n!}
=
\crr+\crr^\prime h+\frac{\crr''h^2}{2}+\frac{\crr'''h^3}{6}+\cdots.
\end{equation}
Indeed, every finite $q$-rational $[x]_q$ is a rational function in~$q$ regular at $q=1$,
and $[x]_1=x$.
The coefficients $\crr^{(n)}$ are regarded as functions 
on the set of ordered quadruples of pairwise distinct rational points.

\begin{thm}
\label{TayPro}
(i)
For every $M\in\PGL(2,\Z)$, one has the identity
\begin{equation}
\label{PGLCrq}
\crr_q\bigl(M(x_1),M(x_2);M(x_3),M(x_4)\bigr)
=
\crr_{q^{\det(M)}}(x_1,x_2;x_3,x_4).
\end{equation}

(ii)
More precisely, for every $n\geq0$, the coefficient $\crr^{(n)}$ of the
Taylor series~\eqref{CrT} satisfies
\begin{equation}
\label{PGLTaylor}
\crr^{(n)}\bigl(M(x_1),M(x_2);M(x_3),M(x_4)\bigr)
=
\det(M)^n\,\crr^{(n)}(x_1,x_2;x_3,x_4).
\end{equation}
\end{thm}

Thus the even coefficients of the $h$-expansion, $\crr^{(2m)}$, are genuine $\PGL(2,\Z)$-invariants,
whereas the odd ones, $\crr^{(2m+1)}$, are relative invariants with character~$\det$.

\begin{rem}
The logarithmic coordinate $h=\log q$ is particularly natural here:
the involution $q\mapsto q^{-1}$ becomes simply $h\mapsto-h$.
Thus the full $\PGL(2,\Z)$-action is diagonal on the coefficients
of~\eqref{CrT}, with parity determined by the order.
The coefficient of order~$0$ is the classical cross-ratio, while the first
coefficient $\crr'$ is already a nontrivial relative invariant.
\end{rem}

\subsection*{Algebraic independence.}
Functions $J_0,\ldots,J_{N-1}$ on a discrete set are called 
algebraically independent over $\C$ if there is no nonzero polynomial $P$ in $N$ variables such that
$$
P(J_0,\ldots,J_{N-1})\equiv0.
$$
If the first $N$ terms of an infinite sequence $J_0,J_1,\ldots$ are algebraically independent for all $N$,
then this sequence is algebraically independent. 

A naive dimension count---which is not applicable to a discrete symmetry group---might 
suggest only finitely many independent relative invariants $\crr^{(n)}$.
In fact, we have the following statement.

\begin{thm}
\label{InvThm}
The infinite sequence of (relative) invariants
$$
\crr,\crr',\crr'',\crr''',\ldots
$$
is algebraically independent over~$\C$.  
\end{thm}

In particular,

\begin{cor}
\label{InvCorol}
The genuine $\PGL(2,\Z)$-invariants $\crr, \crr^{(2)}, \crr^{(4)}, \ldots$ 
are algebraically independent over $\C$.
\end{cor}

\subsection*{The lower-order coefficients.}
For rational quadruples $(x_1,x_2,x_3,x_4)$,
the first two nonclassical coefficients
$\crr'$ and $\crr''$ in the $h = \log q$ expansion~\eqref{CrT} can be calculated explicitly.
The computations use the results of Lasker~\cite{Las}.
We use Lasker's explicit formulas and give short recurrence-based verifications for completeness.

The infinitesimal part of $\crr_q$ can be written in terms of the Pl\"ucker determinants.
\begin{thm}
\label{FirstCFThm}
For a quadruple $x_1,x_2,x_3,x_4$ of pairwise distinct (finite) rationals,
the first coefficient $\crr'$ in~\eqref{CrT} is determined explicitly by
\begin{equation}
\label{CrT1}
\frac{\crr'(x_1,x_2;x_3,x_4)}{\crr(x_1,x_2;x_3,x_4)}
=
\frac{1}{2\Delta_{34}}
\left[
\frac{(\Delta_{14}-\Delta_{13})^2-\Delta_{34}^2}
     {\Delta_{13}\Delta_{14}}
-
\frac{(\Delta_{24}-\Delta_{23})^2-\Delta_{34}^2}
     {\Delta_{23}\Delta_{24}}
\right].
\end{equation}
\end{thm}

Note the importance of the normalization issue.  
The classical formula~\eqref{ClassEq}
is independent of the choice of nonzero lifts $X_i$.
For formula~\eqref{CrT1}, however, this is not the case.
The lifts of the points must be normalized as in~\eqref{LiftEq}.

By Theorem~\ref{TayPro}, the expression~\eqref{CrT1} is a relative
$\PGL(2,\Z)$-invariant with character~$\det$:
\begin{equation}
\label{PGLFirst}
\crr'\bigl(M(x_1),M(x_2);M(x_3),M(x_4)\bigr)
=
\det(M)\,\crr'(x_1,x_2;x_3,x_4).
\end{equation}
Note that $\crr'$ is not invariant with respect to the action of the much larger rational
projective group $\PGL(2,\Q)$.

The second derivative $\crr''$ in~\eqref{CrT} can also be calculated.
However, in contrast with $\crr'$, it is not determined by six normalized Pl\"ucker determinants.

\begin{prop}
\label{PProp}
The coefficient $\crr''$ is not determined by the six signed Pl\"ucker
determinants
\[
(\Delta_{12},\Delta_{13},\Delta_{14},
  \Delta_{23},\Delta_{24},\Delta_{34})
\]
of the normalized primitive lifts.  In particular, there is no formula for
$\crr''$ involving only these six integers.
\end{prop}

Following~\cite{Las}, we will use Apostol's generalized Dedekind sum $s_3(a,b)$ of weight $4$ 
(see~\cite{Apostol} and Section~\ref{ApoSec} below). 
Put
$$
s_i:=s_3(a_i,b_i).
$$
For a pair $i\ne j$, set
\begin{equation}
\Phi_{ij}
:=
\frac{(x_i-x_j)^2}{24}
-
\frac18
\left(
\frac{b_i^{-2}-b_j^{-2}}{x_i-x_j}
\right)^2
-
10\frac{s_i-s_j}{x_i-x_j},
\label{eq:PhiCoord}
\end{equation}
(here we assume $x_i,x_j$ finite rational points).
The second coefficient in~\eqref{CrT} is determined by the following formula.

\begin{thm}
\label{SecondCFThm}
For a quadruple of pairwise distinct finite rational points $(x_1,x_2,x_3,x_4)$,
\begin{equation}
\label{CrT2}
\frac12\left[
\frac{\crr''(x_1,x_2;x_3,x_4)}{\crr(x_1,x_2;x_3,x_4)}
-\left(\frac{\crr'(x_1,x_2;x_3,x_4)}{\crr(x_1,x_2;x_3,x_4)}
\right)^2
\right]
=
\Phi_{13}+\Phi_{24}-\Phi_{14}-\Phi_{23}.
\end{equation}
\end{thm}

Note that the expression on the left-hand side of~\eqref{CrT2}
is precisely the coefficient of~$h^2$ in $\log(\crr_{e^h}/\crr)$.

\subsection*{Comparison to the existing literature.}
We briefly compare this construction with earlier quantum and noncommutative cross-ratios.

Earlier objects called quantum cross-ratios~\cite{CHZ,Lei} arise in noncommutative quantum projective geometry, 
through $q$-deformed coordinate/bracket algebras. 
Our construction is of a different nature: 
it applies the scalar $q$-rational deformation to each classical rational point 
and then takes the ordinary cross-ratio.

Classical cross-ratios occur naturally as cluster $\mathcal{X}$-coordinates, particularly for 
$\PGL_2$ moduli spaces and Teichm\"uller spaces; see~\cite{FC,FG}. 
Quantum cluster and quantum Teichm\"uller theories 
quantize the bracket algebra of Teichm\"uller coordinates and their mutation rules. 
This is conceptually different from the present construction and does not give a scalar deformation 
of $\crr$.
  
A noncommutative analogue of the cross-ratio was considered in~\cite{Ret}.
This construction is a noncommutative cross-ratio over skew fields, 
expressed through quasi-Pl\"ucker coordinates, 
rather than a one-parameter $q$-deformation.

The classical cross-ratio was used in the framework of $q$-deformed rationals in~\cite{JPT}.
It would be interesting to determine whether the $q$-deformed version suggested in this paper
can be applied to the Springborn operations considered in~\cite{JPT}.

\subsection*{Organization.}
This paper is organized as follows.
In Section~\ref{CompSec}, we give a very brief introduction to the notion
of $q$-rational numbers.
We emphasize the property of $\PGL(2,\Z)$-equivariance.
In Section~\ref{ExSec}, we give several concrete examples of $q$-cross-ratios.
In Section~\ref{FrizSec}, we explain the relation to $q$-friezes.
In Section~\ref{GenSec}, we prove Theorems~\ref{TayPro} and~\ref{InvThm}
and discuss the separation and
completeness properties of the $q$-cross-ratio.
In Section~\ref{TDerSec}, we reprove the results of Lasker~\cite{Las}
necessary for our computations.
In Section~\ref{ComputSec}, we prove Theorems~\ref{FirstCFThm} and~\ref{SecondCFThm}.
We discuss open questions,
mixed cross-ratios, complex arguments, 
and possible directions of further research in Section~\ref{SecQuest}.
Finally, in Appendix~\ref{PairTripleApp} we recall the classification of
pairs from~\cite{JSW} and give a classification of triples of points
in $\pP^1(\Q)$ under the $\PGL(2,\Z)$-action.
We then deduce a normal-form classification of ordered quadruples.

\section{Background on $q$-rationals}\label{CompSec}

\subsection{Definition and $\PSL(2,\Z)$ equivariance}

The quantization map
\begin{equation}
\label{qMap}
[\,\cdot\,]_q:\pP^1(\Q)\to\pP^1(\Q(q))
\end{equation}
assigns to every finite rational a rational function of~$q$ which is a quotient of two polynomials in~$\Z[q]$
(while $[\infty]_q=\infty$).
This map is uniquely characterized by the image of one point,
namely
$$
[0]_q=0,
$$
and two recurrence relations
\begin{equation}
\label{RecEq}
[x+1]_q=q[x]_q+1,
\qquad\qquad
\left[-\frac1x\right]_q
=
-\frac1{q[x]_q}.
\end{equation}
For a complete proof of the existence and uniqueness, see~\cite{MGOSur}.

The recurrence relations~\eqref{RecEq} express the $\PSL(2,\Z)$-equivariance identity
for the map~\eqref{qMap}.
The action on $\pP^1(\Q)$ is the standard fractional-linear action.
The action on the space of functions in $q$ is generated by two operators:
\begin{equation}
\label{TSEq}
T_q(F(q))=qF(q)+1,
\qquad\qquad
S_q(F(q))=-\frac{1}{qF(q)}.
\end{equation}
Note that the operators $T_q$ and $S_q$ act on the projective line over $\mathbb Q(q)$ and satisfy the relations
$$
S_q^2=(T_qS_q)^3=\id
$$
and thus indeed generate an action of $\PSL(2,\Z)$.
The fractional-linear transformations $T:=T_1$ and $S:=S_1$
are nothing but the standard generators of the $\PSL(2,\Z)$-action on $\pP^1(\Q)\simeq\Q\cup\{\infty\}$.

\subsection{Explicit formulas}
We recall the explicit continued-fraction formulas from~\cite{MGO1}.
When $x=n>0$ is a nonnegative integer, the $q$-deformation is given by the standard formula
\begin{equation}
\label{qBn}
[n]_{q}=1+q+q^{2}+\cdots+q^{n-1},
\end{equation}
(and $[0]_q=0$)
that goes back to Euler and Gauss.
Note that if the integer $n<0$, then 
$$
[n]_{q}:=\textstyle\frac{1-q^n}{1-q}
$$
 is a Laurent polynomial.

Every rational number~$x$ can be written as a finite continued fraction
$$
x
\quad=\quad
c_0 - \cfrac{1}{c_1 
          - \cfrac{1}{\ddots - \cfrac{1}{c_\ell} } } ,
$$
where $c_j$ are integers such that $c_j\geq2$, except for $c_0$ which is an arbitrary integer.
This is called the negative (or Hirzebruch) continued fraction.
The standard notation for this continued fraction is $x=\NCF{c_0,c_1,\ldots,c_\ell}$.

The $q$-deformation of~$x$ is given explicitly by a simple uniform nonalternating formula
\begin{equation}
\label{qa}
[x]_q
=
[c_0]_q
-\cfrac{q^{c_0-1}}{
  [c_1]_q
  -\cfrac{q^{c_1-1}}{
    [c_2]_q
    -\cfrac{q^{c_2-1}}{
      \ddots
      -\cfrac{q^{c_{\ell-1}-1}}{[c_\ell]_q}
    }
  }
},
\end{equation}
where $[c_i]_q$ are as in~\eqref{qBn}.

If one uses instead the regular continued fraction $x=[a_0;a_1,\ldots,a_k]$, 
with $a_0\in\mathbb Z$ and $a_i\ge1$ for $i\ge1$,
$$
x
\quad=\quad
a_0 + \cfrac{1}{a_1 
          + \cfrac{1}{\ddots +\cfrac{1}{a_{k}} } },
          $$
the formula alternates between $q$ and $q^{-1}$:
\begin{equation}
\label{rqa}
[x]_q
=
[a_0]_q
+\cfrac{q^{a_0}}{
  [a_1]_{q^{-1}}
  +\cfrac{q^{-a_1}}{
    [a_2]_q
    +\cfrac{q^{a_2}}{
      \ddots
      +\cfrac{q^{(-1)^{k-1}a_{k-1}}}{
        [a_k]_{q^{(-1)^k}}
      }
    }
  }
}.
\end{equation}
in every other term.

\subsection{$\PGL(2,\Z)$-equivariance}

Although the quantization map~\eqref{qMap} is determined by the property of $\PSL(2,\Z)$-equivariance,
it also commutes with a certain action of $\PGL(2,\Z)$.

Recall that to define an action of $\PGL(2,\Z)$, we need one more generator,
besides $T$ and $S$.
The simplest choice is the inversion map
$$
\mathcal{I}(x)=\frac{1}{x}.
$$
The $q$-deformed action of $\mathcal{I}$ on $\pP^1(\Q(q))$ is then given by
\begin{equation}
\label{IEq}
\mathcal{I}_q(F(q)):=\frac{1}{F(q^{-1})}.
\end{equation}
The occurrence of the parameter inversion $q\mapsto q^{-1}$ is already visible in~\eqref{rqa}.

The operators $T_q,S_q$, and $\mathcal{I}_q$ generate an action of $\PGL(2,\Z)$ on $\pP^1(\Q(q))$.

\begin{rem}
Another expression for the generator $\mathcal{I}_q$, 
which does not invert the parameter~$q$, was found by Jouteur in~\cite{Jou}.
This $q$-preserving extension of the modular action 
relates the left and right deformations.
This is particularly important for $q$-irrationals
defined as Taylor series in $q$ (see~\cite{MGOexp}), for which 
inversion of $q$ is not well defined.
\end{rem}

We summarize the $\PGL(2,\Z)$-equivariance of $q$-rationals as the formula
\begin{equation}
\label{PGLEq}
\left[M(x)\right]_q=
M_q\bigl(\left[x\right]_q\bigr),
\end{equation}
for every element $M\in\PGL(2,\Z)$, where $M_q$ is the corresponding action on $\pP^1(\Q(q))$.
For more details, see~\cite{MGOSur}.

\subsection{Left $q$-rationals}\label{LeftSec}
The notion of left $q$-rationals was developed in~\cite{BBL}
(it was first observed in~\cite[Remark~3.2]{MGOexp}).
Left $q$-rationals are defined by the map
\begin{equation}
\label{qMapLeft}
[\,\cdot\,]^\flat_q:\pP^1(\Q)\to\pP^1(\Q(q)),
\qquad\qquad
[0]^\flat_q=\frac{q-1}{q},
\end{equation}
commuting with the $\PSL(2,\Z)$-action with the generators~\eqref{TSEq}, i.e.,
satisfying the recurrence relations~\eqref{RecEq}.

\begin{rem}
(a)
Note that a $\PSL(2,\Z)$-equivariant map from
$\pP^1(\Q)$ to $\pP^1(\Q(q))$ can have only two possible values at $0$:
$$
[0]^\flat_q=\frac{q-1}{q},
\qquad\hbox{or}\qquad
[0]_q=0.
$$
Indeed, these are precisely the two fixed points
of the fractional-linear transformation generated by the matrix
$$
L_q=
\begin{pmatrix}
q&0\\
q&1
\end{pmatrix}.
$$
More generally, both right and left $q$-rationals can be viewed (see~\cite{MGOSur})
as fixed points of the $q$-deformations of parabolic elements of $\PSL(2,\Z)$.

(b)
A useful example is
\[
 [\infty]^\flat_q=\frac{1}{1-q},
\]
which follows by applying $S_q$ to $[0]^\flat_q$, and is also the second fixed point
of~$T_q$ besides $\infty$. 
Note that, unlike the right deformation for which $[\infty]_q=\infty$, 
the left deformation of $\infty$ is finite over $\Q(q)$.
\end{rem}

For every $x\in\Q$, the left $q$-rational $[x]^\flat_q$ is,
similarly to the $q$-rational $[x]_q$, a rational function in~$q$ with integer coefficients.
The two notions of $q$-rationals are related by the fractional-linear transition map
\begin{equation}
\label{TransitionMap}
 g_q(z)=\frac{1+q(z-1)}{1+(q-1)z},
\end{equation}
found by Thomas~\cite{Tho},
which satisfies
\begin{equation}
\label{LeftRightTransition}
 g_q\bigl([x]_q\bigr)=[x]^{\flat}_{q^{-1}},
 \qquad\qquad
 g_q\bigl([x]^{\flat}_q\bigr)=[x]_{q^{-1}}.
\end{equation}

For real specializations of the parameter~$q$ such that $0<q<1$ and finite~$x$, one always has the inequality
$$
[x]^\flat_q<[x]_q,
$$
see~\cite{BBL,Eti}.

\section{Examples of $q$-cross-ratios}\label{ExSec}

To illustrate the definition, we compute several examples.

\subsection{Harmonic quadruples}

The elementary example of the classical harmonic
quadruple $(\infty,0;1,-1)$.
The defining recurrences of $q$-rationals give
$$
 [\infty]_q=\infty,\qquad [0]_q=0,\qquad [1]_q=1,
 \qquad [-1]_q=-q^{-1}.
$$
Therefore
\begin{equation}
\label{eq:harmonic-basic}
 \crr_q(\infty,0;1,-1)=-q^{-1}.
\end{equation}

\begin{rem}
Recall that 
the hyperbolic surface associated with the $q$-deformed modular action is an
orbifold with one funnel; see~\cite{JPT}.  
The above example is related to the boundary geodesic length.
Indeed,
Jouteur, Paris-Romaskevich, and Thomas proved that
for~$q>0$, the boundary geodesic of the convex core of the
$\PSL(2,\Z)$-quotient has length 
$$
\ell_q=|\log q|
=\left|\log\left|\crr_q(\infty,0;1,-1)\right|\right|
$$
(for the $\PGL(2,\Z)$-quotient
the length is divided by two).
At $q=1$ the boundary degenerates to the cusp of the classical modular surface
and the length formula has limit $0$; see~\cite[Proposition~3.18]{JPT}.
It would be interesting to explore this relation further.
\end{rem}

\subsection{The Fibonacci family}

Let $F_0=0$, $F_1=1$, and $F_{n+2}=F_{n+1}+F_n$ be the classical Fibonacci sequence.
The quotients
$$
 x_i=\frac{F_{2i}}{F_{2i+1}},\qquad i\geq 1
$$
have been studied in~\cite{MGO1} and~\cite{LMGOV}.
Thus
$$
 x_1=\frac12,\quad x_2=\frac35,\quad x_3=\frac8{13},\quad
 x_4=\frac{21}{34},\quad x_5=\frac{55}{89},\quad
 x_6=\frac{144}{233},\ldots
$$
The negative continued fraction of $x_i$ is
$$
 x_i=\NCF{1,\underbrace{3,\ldots,3}_{i-1\text{ times}},2}.
$$

This sequence of rationals is the orbit of a single element of $\PSL(2,\Z)$:
$$
 x_{i+1}=M(x_i),\qquad
 M(x)=\frac{x+1}{x+2},\qquad
 M=\begin{pmatrix}1&1\\1&2\end{pmatrix}.
$$
Consequently, the $\PSL(2,\Z)$-invariance of $\crr_q$ immediately
implies that
$$
 \crr_q(x_i,x_{i+1};x_{i+2},x_{i+3})=
 \frac{1+2q+3q^2+2q^3+q^4}
 {1+2q+2q^2+2q^3+q^4}
$$
is independent of~$i$.

More generally, define the polynomial sequence $(U_n)$ by $U_{-1}=0, U_0=1$ and
$$
U_n=
[3]_qU_{n-1}-q^2U_{n-2}
$$
for $n\geq1$.
The first polynomials $U_n$ are
\begin{align*}
 U_0={}&1,\\
 U_1={}&1+q+q^2,\\
 U_2={}&1+2q+2q^2+2q^3+q^4,\\
 U_3={}&1+3q+4q^2+5q^3+4q^4+3q^5+q^6,\\
 U_4={}&1+4q+7q^2+10q^3+11q^4+10q^5+7q^6+4q^7+q^8,\\
 U_5={}&1+5q+11q^2+18q^3+24q^4+26q^5
       +24q^6+18q^7+11q^8+5q^9+q^{10}.
\end{align*}

\begin{prop}
\label{prop:fib-four}
For $i<j<k<\ell$, let
$$
 \alpha=j-i,\qquad \beta=k-j,\qquad \gamma=\ell-k.
$$
Then
\begin{equation}\label{eq:fib-four}
 \crr_q(x_i,x_j;x_k,x_\ell)
 =
 \frac{U_{\alpha+\beta-1}U_{\beta+\gamma-1}}
 {U_{\alpha+\beta+\gamma-1}U_{\beta-1}}.
\end{equation}
\end{prop}

\begin{proof}
Put \(z_r=[x_r]_q\).  The classical orbit map is
\[
M(x)=\frac{x+1}{x+2}=TST^2(x).
\]
Consequently, its \(q\)-deformation is
\[
M_q(z)=T_qS_qT_q^2(z)
      =\frac{q^2z+q}{q^2z+q+1},
\]
represented by
\[
A_q=
\begin{pmatrix}
q^2&q\\
q^2&q+1
\end{pmatrix},
\qquad
\operatorname{tr}A_q=[3]_q,
\qquad
\det A_q=q^2.
\]
Since \(x_{r+1}=M(x_r)\), equivariance gives \(z_{r+1}=M_q(z_r)\).

Choose a nonzero vector \(V_1\in\mathbb Q(q)^2\) representing \(z_1\), and set
\[
V_r=A_q^{\,r-1}V_1
\qquad (r\geq1).
\]
The Cayley--Hamilton identity for \(A_q\), together with the recurrence defining \(U_n\), gives
\[
A_q^m=U_{m-1}A_q-q^2U_{m-2}I,
\qquad m\geq1.
\]
Indeed, this is immediate for \(m=1\), and multiplication by \(A_q\) proves the induction step.

For \(s>r\), we therefore obtain
\begin{align*}
\det(V_r,V_s)
&=\det(A_q)^{r-1}\det(V_1,A_q^{s-r}V_1)\\
&=q^{2(r-1)}U_{s-r-1}\det(V_1,A_qV_1),
\end{align*}
because the scalar term in \(A_q^{s-r}\) has zero determinant with \(V_1\).
Using the determinant formula for the cross-ratio, all common factors cancel:
\begin{align*}
\crr_q(x_i,x_j;x_k,x_\ell)
&=
\frac{\det(V_i,V_k)\det(V_j,V_\ell)}
     {\det(V_i,V_\ell)\det(V_j,V_k)}\\
&=
\frac{U_{k-i-1}U_{\ell-j-1}}
     {U_{\ell-i-1}U_{k-j-1}}.
\end{align*}
Finally,
\[
k-i=\alpha+\beta,\quad
\ell-j=\beta+\gamma,\quad
\ell-i=\alpha+\beta+\gamma,\quad
k-j=\beta,
\]
which is exactly~\eqref{eq:fib-four}. 
\end{proof}

Thus the answer depends only on the three gaps between the indices.
The coefficient vectors of the polynomials $U_n$ form the rows of OEIS sequence 
\href{https://oeis.org/A364366}{A364366}, the third row-symmetric
Fibonaccian triangle~\cite{OEIS364366}.

\subsection{A trace formula}

The Fibonacci computation is a special case of a useful general identity.

Let $A\in\PSL(2,\Z)$ and $x\in\pP^1(\Q)$,
and let $x,Ax,A^2x,A^3x$ be pairwise distinct.  

\begin{lem}
\label{TraceLem}
Let $A\in\PGL(2,\K)$ and $x\in\pP^1(\mathbb K)$, choose any invertible matrix representative, 
and assume the four orbit points are distinct.
Then
\begin{equation}
\label{eq:trace-formula}
 \crr(x,Ax;A^2x,A^3x)
 =\frac{(\operatorname{tr}A)^2}
 {(\operatorname{tr}A)^2-\det A}.
\end{equation}
\end{lem}

\begin{proof}
Choose a nonzero lift \(V\in\mathbb K^2\) of \(x\).  By the
Cayley--Hamilton identity,
\[
A^2=(\operatorname{tr}A)A-(\det A)I,
\qquad
A^3=\bigl((\operatorname{tr}A)^2-\det A\bigr)A
     -(\operatorname{tr}A)(\det A)I.
\]
Hence
\[
\begin{aligned}
\crr(x,Ax;A^2x,A^3x)
&=
\frac{\det(V,A^2V)\det(AV,A^3V)}
     {\det(V,A^3V)\det(AV,A^2V)}
\\
&=
\frac{
(\operatorname{tr}A)\det(V,AV)\,
(\operatorname{tr}A)(\det A)\det(V,AV)}
{
\bigl((\operatorname{tr}A)^2-\det A\bigr)\det(V,AV)\,
(\det A)\det(V,AV)}
\\
&=
\frac{(\operatorname{tr}A)^2}
     {(\operatorname{tr}A)^2-\det A}.
\end{aligned}
\]
\end{proof}

Note that the ratio is unchanged under scalar rescaling, so that the above formula
does not depend on the choice of the matrix representing $A\in\PSL(2,\Z)$.
Applying this to $q$-deformed matrices $M_q$ gives a large supply of
examples of $q$-cross-ratios.

\begin{ex}
Consider the case of arithmetic progressions.
For a positive integer $d$ and the translation $T^d(x)=x+d$,
the $q$-deformed translation is given by the fractional-linear action of the matrix
$$
 T_q^d=
 \begin{pmatrix}
 q^d&[d]_q\\[4pt]
 0&1
 \end{pmatrix}.
$$
Therefore every four successive terms of an arithmetic progression with
finite~$x$ and common difference \(d\) satisfy
$$
 \crr_q(x+3d,x+2d;x+d,x)
 =\frac{(1+q^d)^2}{1+q^d+q^{2d}}.
$$
\end{ex}
Note that the result if independent of the starting point~$x$.

\begin{ex}
We also have a family with constant interior partial quotient $c$.
For an integer $c\geq2$, consider
$$
 M_c(x)=\frac{x+c-2}{x+c-1}
 =\begin{pmatrix}1&c-2\\1&c-1\end{pmatrix}\!\cdot x.
$$
Its $q$-deformed action is represented by
$$
 (M_c)_q=
 \begin{pmatrix}
 q^{c-1}&[c-1]_q-1\\[4pt]
 q^{c-1}&[c-1]_q
 \end{pmatrix},
 \qquad
 \operatorname{tr}(M_c)_q=[c]_q,
 \qquad
 \det(M_c)_q=q^{c-1}.
$$
Starting at $x_1=1/2$, its orbit is
$$
 x_i^{(c)}=
\NCF{1,\underbrace{c,\ldots,c}_{i-1\text{ times}},2}.
$$
Applying~\eqref{eq:trace-formula}, we obtain
$$
 \crr_q(x_i^{(c)},x_{i+1}^{(c)};x_{i+2}^{(c)},x_{i+3}^{(c)})
 =\frac{[c]_q^2}{[c]_q^2-q^{c-1}}.
$$
\end{ex}

\subsection{A gap formula for integer quadruples}

Let $x_1>x_2>x_3>x_4$ be integers and write
$$
 \alpha=x_1-x_2,\qquad
 \beta=x_2-x_3,\qquad
 \gamma=x_3-x_4.
$$
Using $[m+n]_q=[m]_q+q^m[n]_q$, all powers of $q$ cancel in the
cross-ratio and one obtains
$$
 \crr_q(x_1,x_2;x_3,x_4)
 =\frac{[\alpha+\beta]_q[\beta+\gamma]_q}
 {[\alpha+\beta+\gamma]_q[\beta]_q}.
 $$
Just as in the Fibonacci family, only the three gaps matter.  
Here are two examples.
$$
 \crr_q(3,2;1,0)=
 \frac{(1+q)^2}{1+q+q^2}
$$
and
$$
  \crr_q(4,3;1,0)=
  \frac{(1+q+q^2)^2}{[4]_q[2]_q}
 =\frac{1+2q+3q^2+2q^3+q^4}
 {1+2q+2q^2+2q^3+q^4}.
$$

\section{$q$-cross-ratios and $q$-friezes}\label{FrizSec}

In this section, we describe the relation between the $q$-deformed cross-ratio 
and $q$-analogues of Coxeter friezes introduced in~\cite{MGOqfrieze}.

\subsection{The Gauss \emph{pentagramma mirificum}}

The story goes back to Gauss~\cite{Gau}.
Gauss's \emph{pentagramma mirificum}
was one of Coxeter's motivations for introducing frieze patterns~\cite{Cox}.

One way to formulate Gauss's observation is the following.
Take five (cyclically ordered) points in~$\R\pP^1$:
$$
x_1,\ldots,x_5,
\qquad\qquad
x_{i+5}:=x_i
$$
and consider their five consecutive cross-ratios:
\begin{equation}
\label{GaCR}
\gr_i:=
\gr(x_i,x_{i+1};x_{i+2},x_{i+3})=
\frac{(x_i-x_{i+3})(x_{i+1}-x_{i+2})}{(x_i-x_{i+1})(x_{i+2}-x_{i+3})}.
\end{equation}
Note that the version of the cross-ratio \eqref{GaCR} is related to~\eqref{ClCR} via
$$
\gr(x_i,x_{i+1};x_{i+2},x_{i+3})=
\frac{1}{\crr(x_i,x_{i+1};x_{i+2},x_{i+3})-1}.
$$

Since the moduli space of ordered configurations of five pairwise distinct points of $\R\pP^1$, 
modulo~$\PGL(2,\R)$ 
$$
 \cM_{0,5}(\R)=\{(x_1,\ldots,x_5)\in(\R\pP^1)^5:
 x_i\ne x_j\text{ for }i\ne j\}/\PGL(2,\R)
$$
has dimension~$2$, the cross-ratios $\gr_i,\,i=1,\ldots,5$ cannot be algebraically independent.
Indeed, they satisfy the Gauss relations
\begin{equation}
\label{GaRel}
\gr_i \gr_{i+1} = \gr_{i+3} + 1,
\end{equation}
where the indices are taken modulo $5$.

Five cross-ratios related by Gauss's relations~\eqref{GaRel} can be organized into a 
five-periodic array of numbers
\begin{equation}
\label{GaFr}
\xymatrix @!0 @R=0.45cm @C=0.45cm
{
\cdots&&1&& 1&&1&&1&&1&&1&&1&&1
 \\
&\cdots&&\gr_1&&\gr_2&&\gr_3&&\gr_4&&\gr_5&&\gr_1&&\gr_2&&\cdots
 \\
&&\cdots&&\gr_4&&\gr_5&&\gr_1&&\gr_2&&\gr_3&&\gr_4&&\gr_5&&\cdots
\\
&\cdots&&1&&1&&1&&1&&1&&1&&1&&\cdots
}
\end{equation}
The frieze rule consists in the requirement that every ``elementary diamond''
in this array satisfies the unimodular rule
\begin{equation}
\label{UniRule}
\xymatrix @!0 @R=0.55cm @C=0.55cm
{
&b\ar@{-}[rd]\ar@{-}[ld]&\\
a&&d\\
&c\ar@{-}[ru]\ar@{-}[lu]&\\
}
\qquad\qquad
\xymatrix @!0 @R=0.55cm @C=0.55cm
{
\\
ad-bc=1.
}
\end{equation}

A Coxeter frieze of width $w$ is a $(w+3)$-periodic array with $w$ 
 nontrivial rows between two rows of~$1$'s, 
 satisfying the same unimodular rule.
Cross-ratios appear in the second nontrivial row of every Coxeter frieze; see~\cite{MGsurvey}.

\subsection{Coxeter's friezes}
A Coxeter frieze is an infinite array whose top and bottom rows consist of $1$'s 
$$
\xymatrix @!0 @R=0.75cm @C=0.75cm
{
\cdots&&1&& 1&&1&&1&&1&&\cdots&&1&&\cdots
 \\
&\cdots&&c_{i}\ar@{.}[rrrrrdddd]&&c_{i+1}&&c_{i+2}&&\cdots&&c_{j-1}&&c_j\ar@{.}[llllldddd]&&\cdots
 \\
&&\cdots&&\;\; c_{i,i+1}&&\;\; c_{i+1,i+2}&&\cdots&&c_{j-2,j-1}&&c_{j-1,j}&&\cdots
 \\
 \\
 \\
&&&&&&&&c_{i,j}\\
&&&&&&\cdots&&\cdots&&\cdots\\
&&\cdots& 1&&1&&1&&1&&\cdots&1&&\cdots
}
$$
which is $n$-periodic in each row, that is $c_{i+n,j+n}=c_{i,j}$.
Each elementary diamond satisfies the unimodular rule~\eqref{UniRule}.

\begin{itemize}
\item
The width of the frieze $w=n-3$ is the number of ``nontrivial'' rows,
i.e., rows between the rows of $1$'s.
\item
The first nontrivial row, $\ldots,c_i,c_{i+1},c_{i+2},\ldots$, is called the quiddity row.
Note also that often rows of $0$'s are added before and after the rows of $1$'s.
\item
A frieze is called ``tame'' if every adjacent $3\times3$ determinant (minor) vanishes.
Tameness is an important condition that we will always assume.
\end{itemize}

Although Coxeter was mostly interested in the case where all entries $c_{i,j}$ are positive integers
(as in the next example), friezes with entries in $\Q,\R,\C$, and other fields also make perfect sense
and have been extensively studied; see~\cite{CH,CHP,MGsurvey,SVRS}.

\begin{ex}
\label{FEx}
The following $7$-periodic frieze with $6$ rows is borrowed from~\cite{CoCo}.
$$
\xymatrix @!0 @R=0.45cm @C=0.45cm
{
\cdots&&1&& 1&&1&&1&&1&&1&&1&&1&&1
 \\
&&&1&&4&&2&&1&&3&&2&&2&&1&&\cdots
 \\
\cdots&&1&&3&&7&&1&&2&&5&&3&&1
\\
&1&&2&&5&&3&&1&&3&&7&&1&&\cdots
 \\
 \cdots&&1&&3&&2&&2&&1&&4&&2&&1
 \\
&1&& 1&&1&&1&&1&&1&&1&&1&&1&&\cdots
}
\qquad\qquad
\xymatrix @!0 @R=0.35cm @C=0.35cm
{
&&&1\ar@{-}[rrd]\ar@{-}[lld]
\\
&3\ar@{-}[ldd]\ar@{-}[rrrr]&&&& 2\ar@{-}[rdd]&\\
\\
2\ar@{-}[rdd]&&&&&& 4\ar@{-}[ldd]\ar@{-}[llllldd]\ar@{-}[llllll]\ar@{-}[llllluu]\\
\\
&2\ar@{-}[rrrr]&&&&1
}
$$
For positive integer friezes, Conway and Coxeter proved~\cite{CoCo} 
that the quiddity row records a triangulation of an $n$-gon: 
$c_i$ is the number of triangles incident to its $i$th vertex.
\end{ex}

\subsection{Triality}

The term ``triality'' refers to the correspondence between the
following three types of objects~\cite{SVRS}:
\begin{enumerate}
\item
tame Coxeter friezes over $\R$ of width $n-3$;
\item
labelled $n$-gons $(x_1,\ldots,x_n)$ in $\RP^1$ such that
$x_i\neq x_{i+1}$ for every $i$ (with indices taken modulo $n$),
considered modulo $\PGL(2,\R)$;
\item
linear difference equations (discrete Sturm--Liouville equations)
\begin{equation}
\label{SturmLiouvilleEq}
V_{i+1}=c_iV_i-V_{i-1}
\end{equation}
with real $n$-periodic coefficients $c_i$ and all solutions
$n$-antiperiodic: $V_{i+n}=-V_i$.
\end{enumerate}
For odd $n$, these three spaces are in bijection.
The condition in the second item requires only consecutive vertices
to be distinct; nonconsecutive vertices may coincide.

For the geometric discussion below, we restrict to positive real
friezes. The corresponding polygons have pairwise distinct vertices
appearing in cyclic order on $\RP^1$.
To construct the polygon associated with a frieze, take two
consecutive diagonals and form the following $n$ quotients:
$$
\left(
x_1,x_2,\ldots,x_n
\right)
=\left(
\frac{0}{1},\frac{1}{c_i},\frac{c_{i+1}}{c_{i,i+1}},\frac{c_{i+1,i+2}}{c_{i,i+2}},\ldots,\frac{1}{0}
\right).
$$
This determines $n$ points in $\RP^1$.
Note that, up to the $\PGL(2,\R)$-action, the resulting $n$-gon is independent of the choice
of the consecutive diagonals.
In Example~\ref{FEx}, we obtain
$$
\left(
x_1,x_2,\ldots,x_7
\right)
=\left(
\frac{0}{1},\frac{1}{4},\frac{2}{7},\frac{1}{3},\frac{1}{2},\frac{1}{1},\frac{1}{0}
\right),
$$
but we can also take the heptagon
$$
\left(
x_1,x_2,\ldots,x_7
\right)
=\left(
\frac{1}{1},\frac{4}{3},\frac{7}{5},\frac{3}{2},\frac{2}{1},\frac{1}{0},\frac{0}{1}
\right).
$$
Choosing another pair of consecutive diagonals produces a projectively equivalent heptagon.

Conversely, given an $n$-gon $(x_1,\ldots,x_n)$ with odd~$n$ and
cyclically ordered pairwise distinct vertices in $\RP^1$.
Choose the antiperiodic lift convention 
$$
X_{i+n}=-X_i
$$
such that $\Delta_{i,i+1}=1$ in the fixed orientation.
Note that these are real, determinant-normalized lifts, 
not the positive-denominator primitive lifts of~\eqref{LiftEq}.

We have the following statement whose proof can be found in~\cite{MGsurvey,SVRS}.

\begin{prop}
\label{CoxProp}
(i)
the entries of the corresponding frieze are calculated using the Pl\"ucker determinants
\begin{equation}
\label{ClassPluFr}
 c_{i,j}
 =\frac{\Delta_{i-2,j}}{\Delta_{i-2,i-1}}
\end{equation}
(under the normalization $\Delta_{i,i+1}=1$ for all $i$).

(ii)
The negative of an anharmonic cross-ratios~$\gr$ satisfy
\begin{equation}
\label{EntriesEq}
 \frac{c_{i+2,\ell}\,c_{j+2,k}}
 {c_{i+2,j}\,c_{k+2,\ell}}=
\gr(x_i,x_j;x_k,x_\ell),
\end{equation}
for all $i<j<k<\ell$.
\end{prop}

In particular, the second nontrivial row consists of cross-ratios
\begin{equation}
\label{CREntriesEq}
c_{i,i+1}=
\gr(x_{i-2},x_{i-1};x_i,x_{i+1}).
\end{equation}
We omit the proof here since we will prove a similar statement in the $q$-deformed case.

Finally, the third ingredient of the triality can be explained as follows.
Entries of every diagonal of the frieze satisfy the equation~\eqref{SturmLiouvilleEq}
with coefficients taken in the quiddity row.
For instance, the entries along the $i$-th diagonal satisfy the Sturm–Liouville equantion
$$
c_{i,j}=c_jc_{i,j-1}-c_{i,j-2}.
$$
The property that all the solutions to this equation are antiperiodic is equivalent to the fact that
the frieze eventually ends with a row of $1$'s.

\subsection{$q$-deformed integer friezes}

The notion of $q$-deformed integer Coxeter frieze patterns was introduced in~\cite{MGOqfrieze}.

A $q$-deformed integer frieze is an array of $n-1$ infinite rows of polynomials, each $n$-periodic,
starting with a row of~$1$'s.
Let $(c_i)$ be the quiddity of an integral Conway--Coxeter frieze, 
equivalently the sequence of numbers of incident triangles of a triangulated $n$-gon.
The quiddity row of a $q$-frieze consists of Euler's $q$-integers $[c_i]_q$, as in~\eqref{qBn}, and
every elementary diamond satisfies the $q$-deformed unimodular rule
\begin{equation}
\label{QL}
C_{i,j-1}(q)\,C_{i+1,j}(q)-C_{i+1,j-1}(q)\,C_{ij}(q)
=
q^{\sum_{k=i}^{j-1}(c_k-1)}.
\end{equation}

A $q$-frieze has the following shape
$$
\xymatrix @!0 @R=0.5cm @C=0.65cm
{
\cdots&&1&& 1&&1&&1&&1&&\cdots&&1&&\cdots
 \\
&\cdots&&{\left[c_{i}\right]_q}\ar@{.}[rrrrrddd]&&{\left[c_{i+1}\right]_q}
&&{\left[c_{i+2}\right]_q}&&\cdots&&{\left[c_{j-1}\right]_q}&&\left[c_j\right]_q\ar@{.}[lllllddd]&&\cdots
\\
 \\
 \\
&&&&&&&&C_{ij}(q)\\
&&&&&&\cdots&&\cdots&&\cdots
}
$$

Starting from the row $(\left[c_{i}\right]_q)$ and using~\eqref{QL},
one can calculate each subsequent row of the $q$-frieze inductively.
Example~\ref{FEx} corresponds to the following $q$-deformed frieze
which was calculated in~\cite{MGOqfrieze}. 

\begin{ex}
\label{qFEx} 
The $q$-deformation of the frieze in Example~\ref{FEx} is
$$
\xymatrix @!0 @R=0.7cm @C=0.7cm
{
\cdots&&1&& 1&&1&&1&&1&&1&&1&&1&&1
 \\
&&&1&&[4]_q&&[2]_q&&1&&[3]_q&&[2]_q&&[2]_q&&1&&\cdots
 \\
\cdots&&1&&q[3]_q&&\{7\}_{q}&&1&&q[2]_q&&\{5\}_{q}&&[3]_{q}&&1
\\
&1&&q^2[2]_q&&q\{5\}_{q}&&[3]_q&&q^2&&q[3]_q&&\{7\}_{q}&&1&&\cdots
 \\
 \cdots&&q^3&&q^2[3]_q&&q[2]_q&&q^2[2]_q&&q^3&&q[4]_q&&[2]_q&&q^3
 \\
&q^3&&q^4&& q^2&&q^3&&q^3&&q^4&&q&&q^3&&q^4&&\cdots
}
$$
where 
$$
\{7\}_{q}=1+2q+2q^2+q^3+q^4
\qquad\hbox{and}\qquad
\{5\}_{q}=1+2q+q^2+q^3.
$$
\end{ex}

A $q$-frieze ends with a row of powers of $q$.

\subsection{The projective $q$-polygon hidden in a $q$-frieze}

At $q=1$, the $q$-frieze specializes to an integral
Conway--Coxeter frieze. The associated rational $n$-gon $(x_1,x_2,\ldots,x_n)$ in $\pP^1(\Q)$ has
cyclically ordered, pairwise distinct vertices, but is initially
determined only up to projective equivalence. We now fix a
representative using the quiddity and an initial unimodular pair.
This integral normalization matters: the equivariance of
$x\mapsto[x]_q$ is with respect to the modular group, not the full
group $\PGL(2,\Q)$.

Choose an ear of the triangulation and index the quiddity so that
$c_0=1$. Fix the initial unimodular pair
\[
V_{-1}=\begin{pmatrix}1\\0\end{pmatrix},
\qquad
V_0=\begin{pmatrix}1\\1\end{pmatrix},
\]
and define $V_j=\begin{pmatrix}
\Rc_j(q)\\ \Sc_j(q)
\end{pmatrix}$ by
$$
V_j=[c_j]_qV_{j-1}-q^{c_{j-1}-1}V_{j-2},
$$
where $[c_j]_q$ are the entries of the $n$-periodic quiddity row of the $q$-frieze.
Set 
\[
x_j=\NCF{1,c_1,\ldots,c_j}=\Rc_j(1)/\Sc_j(1).
\]
The recurrence identifies $\Rc_j(q)/\Sc_j(q)$ with the canonical $q$-rational $[x_j]_q$:
\[
[x_j]_q=\frac{\Rc_j(q)}{\Sc_j(q)}.
\]
For these vertices, define the $q$-deformed Pl\"ucker determinants by
\[
\Delta^q_{i,j}=\det\bigl(V_i(q),V_j(q)\bigr).
\]
Here the vectors are those given by the recurrence, without
individual rescaling.

We have the following $q$-analogue of Proposition~\ref{CoxProp}.

\begin{prop}
\label{thm:pluecker-frieze}
For all indices in the frieze,
\begin{equation}
\label{eq:pluecker-entry}
 C_{i,j}(q)
 =\frac{\Delta^q_{i-2,j}}{\Delta^q_{i-2,i-1}},
\end{equation}
where $\Delta^q_{i,j}=\det(V_i(q),V_j(q))$; at $q=1$ the vectors specialize to primitive integer lifts of the rational vertices.
\end{prop}

\begin{proof}
The entries of the $q$-frieze 
satisfy the recurrence
\begin{equation}
\label{eq:right-recurrence}
 C_{i,j}
 =[c_j]_qC_{i,j-1}
  -q^{c_{j-1}-1}C_{i,j-2},
 \qquad j\geq i,
\end{equation}
see~\cite{MGOqfrieze}.  
The boundary values are
$$
C_{i,i-2}(q)=0,
 \qquad
 C_{i,i-1}(q)=1,
 \qquad
 C_{i,i}(q)=[c_i]_q.
$$
It is now easy to check using
$$
\det(V_{i-2},V_j)
=[c_j]_q\det(V_{i-2},V_{j-1})
-q^{c_{j-1}-1}\det(V_{i-2},V_{j-2}),
$$
that the right-hand side of~\eqref{eq:pluecker-entry}
also satisfies~\eqref{eq:right-recurrence}, as well as the three boundary values above.
\end{proof}

Proposition~\ref{thm:pluecker-frieze} implies that
the $q$-diamond rule reads
\begin{equation}
\label{qDiam}
 C_{i,j-1}C_{i+1,j}-C_{i+1,j-1}C_{i,j}
 =
 \frac{\Delta^q_{j-1,j}}{\Delta^q_{i-1,i}}.
\end{equation}
Indeed, to deduce this formula from~\eqref{eq:pluecker-entry},
one just applies the ordinary Pl\"ucker relation
$$
 \Delta^q_{i-2,j-1} \Delta^q_{i-1,j}
 - \Delta^q_{i-1,j-1} \Delta^q_{i-2,j}
 = \Delta^q_{i-2,i-1} \Delta^q_{j-1,j}
$$
for the polynomial vectors $V_i$ defined above.
Furthermore, the recurrence~\eqref{eq:right-recurrence} implies that the consecutive
Pl\"ucker determinants satisfy
\begin{equation}
\label{ConsPD}
 \Delta^q_{j,j+1}=
 q^{c_j-1} \Delta^q_{j-1,j}.
\end{equation}
Substituting this equation into~\eqref{qDiam} gives the diamond rule~\eqref{QL}.

\subsection{Relation to $q$-cross-ratio}
We are ready to formulate the main result of this section
which is a more precise formulation of Theorem~\ref{CoxThmIntro}.
Recall that
$$
\gr_q(x_1,x_2;x_3,x_4)
:=
\frac{1}{\crr_q(x_1,x_2;x_3,x_4)-1}
=
\frac{([x_1]_q-[x_4]_q)([x_2]_q-[x_3]_q)}
     {([x_1]_q-[x_2]_q)([x_3]_q-[x_4]_q)}.
$$

\begin{thm}
\label{thm:second-row}
For the $q$-frieze and associated cyclically ordered $n$-gon constructed above, 
with $n\geq5$, one has
\begin{equation}
\label{eq:main-chi}
 C_{i,i+1}(q)=q^{c_i-1}
 \gr_q(x_{i-2},x_{i-1};x_i,x_{i+1}).
\end{equation}
\end{thm}

\begin{proof}
By Proposition~\ref{thm:pluecker-frieze},
\begin{align*}
 \gr_q(x_{i-2},x_{i-1};x_i,x_{i+1})
 &=\frac{\Delta^q_{i-2,i+1}\Delta^q_{i-1,i}}
 {\Delta^q_{i-2,i-1}\Delta^q_{i,i+1}}\\
 &=C_{i,i+1}(q)\frac{\Delta^q_{i-1,i}}{\Delta^q_{i,i+1}}
 =q^{1-c_i}C_{i,i+1}(q),
\end{align*}
where the last equality is~\eqref{ConsPD}.
\end{proof}

The same argument describes not only four consecutive points but every
quadruple of vertices of the projective $n$-gon.

\begin{cor}
\label{cor:general-quadruple}
Let $i<j<k<\ell<i+n$, and use the boundary convention $C_{i,i-1}=1$.
Then one has
\begin{equation}
\label{eq:general-chi}
 \gr_q(x_i,x_j;x_k,x_\ell)
 =q^{-\sum_{s=j+1}^{k}(c_s-1)}
 \frac{C_{i+2,\ell}\,C_{j+2,k}}
 {C_{i+2,j}\,C_{k+2,\ell}}.
\end{equation}
\end{cor}

\begin{proof}
Formula~\eqref{eq:pluecker-entry} gives
$$
 \Delta^q_{i,\ell}=C_{i+2,\ell}\Delta^q_{i,i+1},
 \quad
 \Delta^q_{j,k}=C_{j+2,k}\Delta^q_{j,j+1},
 \quad
 \Delta^q_{i,j}=C_{i+2,j}\Delta^q_{i,i+1},
 \quad
 \Delta^q_{k,\ell}=C_{k+2,\ell}\Delta^q_{k,k+1}.
$$
Substitution into~\eqref{GaCR} leaves the factor
$$
 \frac{\Delta^q_{j,j+1}}{\Delta^q_{k,k+1}}
 =q^{-\sum_{s=j+1}^{k}(c_s-1)},
$$
and hence~\eqref{eq:general-chi}.
\end{proof}

\begin{ex}
Consider for instance Example~\ref{qFEx}.
Recall that the quiddity of the considered frieze is $(1,4,2,1,3,2,2)$.
The entry $\{7\}_q$ in this example can be calculated as follows.
Recurrence~\eqref{eq:right-recurrence} gives
$$
 V_{-1}=\binom10,
\quad
V_0=\binom11,
\quad
 V_1=\binom{q[3]_q}{[4]_q},
 \quad
 V_2=\binom{q[3]_q[2]_q-q^3}{\{7\}_q}.
$$
At $q=1$, the corresponding projective points are
$$
 x_{-1}=\infty,
 \qquad x_0=1,
 \qquad x_1=\frac34,
 \qquad x_2=\frac57.
$$
The relevant determinants are
\begin{gather*}
 \Delta^q_{-1,0}=1,
 \qquad  \Delta^q_{0,1}=1,
 \qquad  \Delta^q_{1,2}=q^3,
 \qquad  \Delta^q_{-1,2}=\{7\}_q,\\
  \Delta^q_{-1,1}=[4]_q,
 \qquad
  \Delta^q_{0,2}=[2]_q.
\end{gather*}
Therefore
\begin{equation}\label{eq:example-seven}
 \gr_q(x_{-1},x_0;x_1,x_2)
 =\frac{\{7\}_q}{q^3},
\end{equation}
in full accordance with~\eqref{eq:main-chi}.
\end{ex}

\section{Modular covariance, algebraic independence, and separation}\label{GenSec}

We now prove the covariance and algebraic independence results, Theorems~\ref{TayPro} and~\ref{InvThm}.

\subsection{Proof of Theorem~\ref{TayPro}}

The proof uses the full $\PGL(2,\Z)$-equivariance of the quantization
map~\eqref{qMap}.  
The elements of $\PSL(2,\Z)$ act by ordinary
fractional-linear transformations over $\Q(q)$, 
whereas the extra generator
$\mathcal I$ replaces $q$ by~$q^{-1}$.  
This is exactly the source of the
transformation law~\eqref{PGLCrq}.

We first prove~\eqref{PGLCrq}.  If $M\in\PSL(2,\Z)$, then $M_q$ is a
fractional-linear transformation.  
Recall that the $q$-deformed generators are
$$
T_q(z)=qz+1,
\qquad
S_q(z)=-\frac1{qz},
$$
and hence they are represented, over the field $\Q(q)$, by the matrices
$$
\begin{pmatrix}q&1\\0&1\end{pmatrix},
\qquad
\begin{pmatrix}0&-1\\q&0\end{pmatrix},
$$
respectively.  Thus $M_q\in\PGL(2,\Q(q))$.  
Applying equivariance to all
four arguments and using the classical invariance of the cross-ratio gives
\begin{equation}
\label{PSLCrqProof}
\crr_q\bigl(M(x_1),M(x_2);M(x_3),M(x_4)\bigr)
=
\crr_q(x_1,x_2;x_3,x_4).
\end{equation}
Since this is an identity of rational functions, the corresponding Taylor coefficients agree.

Now let $M\in\PGL(2,\Z)$ have determinant $-1$.  
Write
$$
M=N\circ\mathcal I,
\qquad N\in\PSL(2,\Z),
\qquad \mathcal I(x)=\frac1x.
$$
By~\eqref{PGLEq} and the definition of $\mathcal I_q$,
$$
[M(x)]_q
=N_q\!\left(\frac1{[x]_{q^{-1}}}\right).
$$
Since $N_q$ is fractional-linear, its simultaneous action on the four arguments cancels by projective invariance.  
The ordinary inversion $z\mapsto1/z$ also leaves the
classical cross-ratio invariant.  
Therefore
\begin{align*}
&\crr_q\bigl(M(x_1),M(x_2);M(x_3),M(x_4)\bigr)
\\
&\quad=
\crr\left(
\frac1{[x_1]_{q^{-1}}},\frac1{[x_2]_{q^{-1}}};
\frac1{[x_3]_{q^{-1}}},\frac1{[x_4]_{q^{-1}}}
\right)
\\
&\quad=
\crr\bigl([x_1]_{q^{-1}},[x_2]_{q^{-1}};
          [x_3]_{q^{-1}},[x_4]_{q^{-1}}\bigr)
\\
&\quad=
\crr_{q^{-1}}(x_1,x_2;x_3,x_4).
\end{align*}
Together with~\eqref{PSLCrqProof}, this proves~\eqref{PGLCrq} for all
$M\in\PGL(2,\Z)$.

It remains to translate this identity into Taylor coefficients.
Put $q=e^h$ and use~\eqref{CrT}.  Formula~\eqref{PGLCrq} becomes
\[
\crr_{e^h}\bigl(M(x_1),M(x_2);M(x_3),M(x_4)\bigr)
=
\crr_{e^{\det(M)h}}(x_1,x_2;x_3,x_4).
\]
Expanding both sides in powers of~$h$, we obtain
$$
\sum_{n\geq0}
\frac{
\crr^{(n)}\bigl(M(x_1),M(x_2);M(x_3),M(x_4)\bigr)h^n
}
{n!}
=
\sum_{n\geq0}
\frac{
\det(M)^n\crr^{(n)}(x_1,x_2;x_3,x_4)h^n
}
{n!}.
$$
Comparison of coefficients gives~\eqref{PGLTaylor}, and~\eqref{PGLFirst}
is the case $n=1$.

Theorem~\ref{TayPro} is proved.

\subsection{Proof of Theorem~\ref{InvThm}}

Fix $N\geq1$.  We prove that the first $N+1$ coefficients
$
\crr,\crr',\ldots,\crr^{(N)}
$
are algebraically independent.

Consider the special family of quadruples
\[
(\infty,0;1,x),
\qquad x\in\Q\setminus\{0,1\}.
\]
Since $[\infty]_q=\infty$, $[0]_q=0$, and $[1]_q=1$, we have
\begin{equation}
\label{SpecialCr}
\crr_{e^h}(\infty,0;1,x)=[x]_{e^h}.
\end{equation}
It is enough to prove algebraic independence after restriction to
this family.

Allow now the coefficients $c_0,\ldots,c_N$ of a negative continued fraction to
vary near $(1,2,\ldots,2)$ in~$\C^{N+1}$. 
We interpolate the finite continued-fraction coefficients, prove that the resulting map to the first $N+1$ 
Taylor coefficients has a nonzero Jacobian, and then use the integer specializations to exclude polynomial relations.
For this interpolation (only), put
\[
[c]_{e^h}:=\frac{e^{ch}-1}{e^h-1},
\]
where the value at $h=0$ is obtained by continuity.
This interpolation is only an auxiliary analytic device; 
it should not be confused with the theories of \(q\)-deformed real or complex numbers.

Let $c_0,\ldots,c_N$ be $N+1$ variables and define recursively
\begin{equation}
\label{InterpolatingCF}
F_N(h)=[c_N]_{e^h},
\qquad
F_j(h)=[c_j]_{e^h}
-\frac{e^{(c_j-1)h}}{F_{j+1}(h)}
\quad(0\leq j<N).
\end{equation}
When $c_0\in\Z$ and $c_1,\ldots,c_N\geq2$ are integers, formula~\eqref{qa} shows that
$F_0(h)=[x]_{e^h}$ for
\[
x=\NCF{c_0,c_1,\ldots,c_N}.
\]
Write the Taylor expansion
\begin{equation}
\label{JetExpansion}
F_0(h)=:\sum_{i\geq0}a_i(c_0,\ldots,c_N)h^i.
\end{equation}
For every fixed $i$, the coefficient $a_i$ is a rational function of
$c_0,\ldots,c_N$, regular in a neighborhood of $(1,2,\ldots,2)$.

\begin{lem}
\label{JacLem}
The Jacobian of $a_0,\ldots,a_N$ in~\eqref{JetExpansion} with respect to
$c_0,\ldots,c_N$ is nonzero at $(1,2,\ldots,2)$.  
\end{lem}

\begin{proof}
Put $q=e^h$.
At this point $(1,2,\ldots,2)$, a direct induction in~\eqref{InterpolatingCF} gives
\begin{equation}
\label{SpecialTail}
F_j(h)=\frac{[N-j+2]_q}{[N-j+1]_q},
\qquad 1\leq j\leq N,
\end{equation}
and
\[
F_0(h)=1-\frac{[N]_q}{[N+1]_q}
      =\frac{q^N}{[N+1]_q}.
\]
For $j\geq1$, set temporarily $r=N-j+1$.  Differentiation at the $j$-th
level gives
\[
\left.\frac{\partial F_j}{\partial c_j}\right|_{(1,2,\ldots,2)}
=
\frac{h q\bigl([r-1]_q+q^r\bigr)}{(q-1)[r]_q},
\]
while propagation from the $j$-th level to $F_0$ contributes the factor
\[
q^{j-1}\frac{[r]_q^2}{[N+1]_q^2}.
\]
For $j=0$, one differentiates $F_0$ directly.  In both cases one obtains
\begin{equation}
\label{DerivativeCF}
\left.
\frac{\partial F_0(h)}{\partial c_j}
\right|_{(1,2,\ldots,2)}
=
\frac{h q^{N+1}}{(q-1)[N+1]_q^2}\,
q^{-r}[r]_q\bigl([r-1]_q+q^r\bigr),
\qquad r=N-j+1.
\end{equation}

The first factor in~\eqref{DerivativeCF} is independent of $j$ and has
constant term $1/(N+1)^2$.  Multiplication by such a power series is an
invertible lower triangular operation on Taylor coefficients.  Consequently, the Jacobian in question is nonzero if and only if the
$(N+1)\times(N+1)$ matrix formed by the coefficients of
$h^0,h^1,\ldots,h^N$ in the series
\[
q^{-r}[r]_q\bigl([r-1]_q+q^r\bigr),
\qquad r=1,2,\ldots,N+1,
\]
has nonzero determinant.

Multiplication of all these series by the common unit
$(e^h-1)^2/h^2$ again does not change the nonvanishing of the determinant.
A direct calculation gives
\begin{equation}
\label{ElementarySeries}
\frac{(e^h-1)^2}{h^2}\,
q^{-r}[r]_q\bigl([r-1]_q+q^r\bigr)
=
\frac{
 e^{(r+1)h}-e^{rh}+e^{(r-1)h}+e^{-rh}-e^h-e^{-h}
}{h^2}.
\end{equation}
The coefficient of $h^i$ in the right-hand side is the polynomial
\begin{equation}
\label{CoefficientPolynomial}
\frac{
(r+1)^{i+2}-r^{i+2}+(r-1)^{i+2}+(-r)^{i+2}
-1-(-1)^{i+2}
}{(i+2)!}.
\end{equation}
For $i=2k$, this is
\[
\frac{(r+1)^{2k+2}+(r-1)^{2k+2}-2}{(2k+2)!};
\]
it has degree $2k+2$, leading coefficient $2/(2k+2)!$, and contains
only even powers of $r$.  For $i=2k+1$, expression~\eqref{CoefficientPolynomial}
becomes
\[
\frac{(r+1)^{2k+3}+(r-1)^{2k+3}-2r^{2k+3}}{(2k+3)!};
\]
it has degree $2k+1$, leading coefficient $1/(2k+1)!$, and contains
only odd powers of $r$.

Suppose first that $N=2m-1$.  After reordering the rows by degree, the
polynomials~\eqref{CoefficientPolynomial} have degrees
\[
1,2,\ldots,2m.
\]
All of them vanish at $r=0$.  
Factoring the value $r_j$ from the $j$th evaluation column, their
quotients have consecutive degrees $0,1,\ldots,2m-1$ and nonzero leading
coefficients.  Their evaluation determinant at the distinct points
$r=1,2,\ldots,2m$ is therefore a nonzero multiple of the ordinary
Vandermonde determinant.

Now suppose that $N=2m$.  The first $2m$ rows give, after the same division
by $r$, a basis of the polynomials of degrees at most $2m-1$.  The remaining
row, corresponding to $i=2m$, has degree $2m+1$ after division by $r$ and
has no term of degree $2m$.  Eliminating its lower-degree terms reduces the
determinant to a nonzero multiple of
\[
\det\bigl(r_j^{d}\bigr)_{
 d=0,1,\ldots,2m-1,\,2m+1}^{\quad j=1,\ldots,2m+1}
=
(r_1+\cdots+r_{2m+1})
\prod_{a<b}(r_b-r_a),
\]
where $r_j=j$.  This determinant is again nonzero.  Thus, for every $N$,
\begin{equation}
\label{NonzeroJetJacobian}
\det\left(
\frac{\partial a_i}{\partial c_j}
\right)_{0\leq i,j\leq N}
\Bigg|_{(1,2,\ldots,2)}
\neq0.
\end{equation}
\end{proof}

We can now conclude the proof.  Suppose that a polynomial
$P(T_0,\ldots,T_N)$ satisfies
\[
P(\crr,\crr',\ldots,\crr^{(N)})=0
\]
for every rational quadruple.  
By~\eqref{JetExpansion}, \(a_i\) is the coefficient of \(h^i\), whereas \(\crr^{(i)}\) is the \(i\)-th derivative.  
Hence
\[
\crr^{(i)}=i!\,a_i
\]
on the family~\eqref{SpecialCr}.
By~\eqref{SpecialCr}, for all integral
$c_0,\ldots,c_N\geq2$ we then have
\[
P\bigl(a_0,1!a_1,2!a_2,\ldots,N!a_N\bigr)=0.
\]
The left-hand side is a rational function of the $c_j$.
After clearing denominators, we obtain a polynomial that vanishes
on a Cartesian product of infinite subsets of $\Z$, and hence
vanishes identically.
Thus the original rational function vanishes identically,
in particular in a neighborhood of $(1,2,\ldots,2)$.
But~\eqref{NonzeroJetJacobian} and the local inverse function theorem show
that the map
\[
(c_0,\ldots,c_N)\longmapsto(a_0,\ldots,a_N)
\]
has an open image near that point.  Hence $P$ vanishes on an open subset of
$\C^{N+1}$, and so $P=0$.  This proves the algebraic independence of
$\crr,\crr',\ldots,\crr^{(N)}$.  Since $N$ is arbitrary,
Theorem~\ref{InvThm} follows.

\subsection{Fixed-triple separation and infinite fibers}
\label{subsec:completeness}
In this section, we demonstrate that $\crr_q$ is not a complete invariant
of quadruples of points of~$\pP^1(\Q)$.

We first prove an elementary separation property.

\begin{prop}
\label{prop:fixed-triple-separation}
Let $(x_1,x_2,x_3)$ be a fixed ordered triple of pairwise distinct points of
$\pP^1(\Q)$.  Then the map
\[
 x_4\longmapsto \crr_q(x_1,x_2;x_3,x_4)
\]
is injective on $\pP^1(\Q)\setminus\{x_1,x_2,x_3\}$.
\end{prop}

\begin{proof}
Suppose that
\[
 \crr_q(x_1,x_2;x_3,x_4)
 =\crr_q(x_1,x_2;x_3,y_4).
\]
Evaluation at $q=1$ gives
\[
 \crr(x_1,x_2;x_3,x_4)
 =\crr(x_1,x_2;x_3,y_4).
\]
For fixed pairwise distinct $x_1,x_2,x_3$, the classical cross-ratio is a
fractional-linear, hence injective, function of its fourth argument.
Therefore $x_4=y_4$.
\end{proof}

The global situation is very different.  Recall that the classical Klein
four-group
\[
 V_4=\{1,(12)(34),(13)(24),(14)(23)\}\subset S_4
\]
fixes the cross-ratio.  Since $\crr_q$ is the ordinary cross-ratio of the four
$q$-rational points, it is also invariant under~$V_4$.  It therefore descends
to the quotient by the joint action of $\PSL(2,\Z)$ and~$V_4$.  For the full
$\PGL(2,\Z)$-action, the corresponding invariant is the unordered pair
$\{\crr_q,\crr_{q^{-1}}\}$.

\begin{prop}
\label{prop:infinite-klein-fiber}
For every integer $n\geq2$, set
\[
 X_n=\left(-n,-\frac1n;1,-1\right),
 \qquad
 Y=(-1,1;\infty,0).
\]
Then
\begin{equation}
\label{eq:infinite-klein-fiber}
 \crr_q(X_n)=\crr_q(Y)=-q.
\end{equation}
Moreover, the quadruples
\[
 Y,X_2,X_3,\ldots
\]
are pairwise non-equivalent even under the joint action of
$\PGL(2,\Z)$ and the full symmetric group~$S_4$.
\end{prop}

\begin{proof}
The recurrence relations~\eqref{RecEq} give
\[
 [-n]_q=-q^{-n}[n]_q,
 \qquad
 \left[-\frac1n\right]_q=-\frac1{q[n]_q},
 \qquad
 [-1]_q=-q^{-1}.
\]
It follows that
\begin{align*}
 [-n]_q-1&=-q^{-n}[n+1]_q,
 &
 \left[-\frac1n\right]_q-[-1]_q&=\frac{[n-1]_q}{[n]_q},\\
 [-n]_q-[-1]_q&=-q^{-n}[n-1]_q,
 &
 \left[-\frac1n\right]_q-1&=-\frac{[n+1]_q}{q[n]_q}.
\end{align*}
Substitution in the definition of the $q$-cross-ratio yields
\begin{align*}
 \crr_q(X_n)
 &=
 \frac{
 \bigl(-q^{-n}[n+1]_q\bigr)
 \dfrac{[n-1]_q}{[n]_q}}
 {
 \bigl(-q^{-n}[n-1]_q\bigr)
 \left(-\dfrac{[n+1]_q}{q[n]_q}\right)}
 =-q.
\end{align*}
Using the usual projective convention when one of the points is~$\infty$,
we also have
\[
 \crr_q(Y)
 =\frac{[1]_q-[0]_q}{[-1]_q-[0]_q}
 =-q.
\]
This proves~\eqref{eq:infinite-klein-fiber}.

For an ordered quadruple $X=(x_1,x_2,x_3,x_4)$, let
\[
 \mathcal D(X)
 =\bigl\{|\Delta_{ij}|:1\leq i<j\leq4\bigr\}
\]
denote the multiset of the six absolute Pl\"ucker determinants of the
primitive lifts~\eqref{LiftEq}.  This multiset is invariant under
$\PGL(2,\Z)$, while a permutation of the four points merely permutes its
entries.  Direct calculation gives
\[
 \mathcal D(X_n)
 =\{n^2-1,n+1,n+1,n-1,n-1,2\}
\]
and
\[
 \mathcal D(Y)=\{2,1,1,1,1,1\}.
\]
Thus $X_n$ is not equivalent to~$Y$.  Furthermore, the largest member of
$\mathcal D(X_n)$ is $n^2-1$, which strictly increases with~$n$.
Hence the quadruples $X_n$ are pairwise non-equivalent, even after arbitrary
relabeling.  Finally, all the above quadruples have the same unordered pair
$\{-q,-q^{-1}\}$.
\end{proof}

Proposition~\ref{prop:infinite-klein-fiber} implies the following.

\begin{cor}
The fiber over $-q$ contains infinitely many orbits
for the joint action of $\PSL(2,\Z)$ and~$V_4$, and the invariant
$\{\crr_q,\crr_{q^{-1}}\}$ is not complete for the joint action of
$\PGL(2,\Z)$ and~$V_4$.
\end{cor}

\begin{rem}
The preceding family has classical cross-ratio $-1$, but the failure of
completeness is not confined to the harmonic case.  For every $n\geq3$,
\[
 \crr_q(-n,1;0,\infty)
 =
 \crr_q\left(-n,-\frac1n;-1,\infty\right)
 =[-n]_q.
\]
The two corresponding multisets of absolute Pl\"ucker determinants are
\[
 \{n+1,n,1,1,1,1\}
 \qquad\hbox{and}\qquad
 \{n^2-1,n,n-1,n-1,1,1\}.
\]
Their largest members are different for $n\geq3$, so the two quadruples are
not equivalent even after arbitrary relabeling.  At $q=1$, their common
cross-ratio is $-n$.
\end{rem}

Proposition~\ref{prop:infinite-klein-fiber} is fully compatible with
Theorem~\ref{InvThm}.  Algebraic independence concerns the invariant
functions as the quadruple varies; it does not imply point separation.
Indeed, on the infinite family above the whole rational function $\crr_q$ is
identical, and therefore all of its logarithmic Taylor coefficients at
$q=1$ are identical as well.

\section{The first two derivatives of $[x]_q$ at $q=1$}\label{TDerSec}

Let $x$ be a rational number
$$
x=\frac ab,
\qquad
\gcd(a,b)=1,
\qquad
b>0.
$$
We will need explicit formulas for the first two derivatives of $[x]_q$ at $q=1$.
These formulas were obtained by Lasker~\cite{Las} using Stern-Brocot trees, 
we give short proofs based on the defining recurrences.

Let us use the following notation:
$$
D_1(x):=\left.\frac{d}{dq}[x]_q\right|_{q=1},
\qquad\qquad
D_2(x):=\left.\frac{d^2}{dq^2}[x]_q\right|_{q=1}.
$$

\subsection{The first derivative}

\begin{thm}
[\cite{Las}, Theorem 1.1] 
\label{FDerLem}
The first derivative of $[x]_q$ at $q=1$ is
\begin{equation}
\label{FDEq}
D_1(x)
=
\frac{x^2-x+1}{2}
-
\frac{1}{2b^2}.
\end{equation}
\end{thm}

\begin{proof}
Differentiating the first recurrence in~\eqref{RecEq} gives
$$
\frac{d}{dq}[x+1]_q
=
[x]_q+q\frac{d}{dq}[x]_q.
$$
Evaluating at $q=1$ and using $[x]_1=x$, we obtain
\begin{equation}
\label{FirstEq}
D_1(x+1)
=
D_1(x)+x.
\end{equation}

Similarly, differentiating the second recurrence in~\eqref{RecEq} one obtains
$$
\frac{d}{dq}\left[-\frac1x\right]_q=
\frac{1}{q^2[x]_q}+\frac{([x]_q)'}{q[x]_q^2}.
$$
At $q=1$ this gives
\begin{equation}
\label{SecEq}
D_1\!\left(-\frac1x\right)
=
\frac1x+\frac{D_1(x)}{x^2}.
\end{equation}
Relations \eqref{FirstEq}--\eqref{SecEq} together with 
the initial value $D_1(0)=0$ determine $D_1$ on $\Q$.

Define, for $x=a/b$
$$
F\!\left(\frac ab\right)
:=
\frac12\left(
x^2-x+1-\frac1{b^2}
\right).
$$
Clearly, $F(0)=0$.
Let us show that $F$ satisfies the same relations \eqref{FirstEq}--\eqref{SecEq}.

(a)
Since
$$
x+1=\frac{a+b}{b},
$$
the reduced denominator remains $b$ because $\gcd(a+b,b)=\gcd(a,b)=1$. 
Therefore
\begin{align*}
F(x+1)-F(x)
&=
\frac12\left(
(x+1)^2-(x+1)-x^2+x
\right)
\\
&=x.
\end{align*}
Thus $F$ satisfies~\eqref{FirstEq}.

(b)
One has
$$
F\!\left(-\frac1x\right)
=
\frac12\left(
\frac1{x^2}+\frac1x+1-\frac1{a^2}
\right).
$$
On the other hand, 
\begin{align*}
\frac1x+\frac{F(x)}{x^2}
&=
\frac1x
+
\frac1{2x^2}
\left(
x^2-x+1-\frac1{b^2}
\right)
\\
&=
\frac1x
+
\frac12
\left(
1-\frac1x+\frac1{x^2}
-\frac1{b^2x^2}
\right)
\\
&=
F\!\left(-\frac1x\right).
\end{align*}
Thus \(F\) satisfies the inversion relation~\eqref{SecEq}.

We have proved that $F$ and $D_1$ satisfy the same defining relations and agree at $0$.
Hence $F(x)=D_1(x)$ for all~$x$.
\end{proof}

\subsection{The second derivative: recurrence relation}

Let us start with the recurrence relation.

\begin{lem}
The first of the following relations holds for every rational $x$, and the second for $x\ne0$:
\begin{align}
D_2(x+1)&=D_2(x)+2D_1(x),
\label{eq:D2T}\\[4pt]
D_2\!\left(-\frac1x\right)
&=
-\frac2x
-\frac{2D_1(x)}{x^2}
+\frac{D_2(x)}{x^2}
-\frac{2D_1(x)^2}{x^3}.
\label{eq:D2S}
\end{align}
Moreover $D_2(0)=0$.
\end{lem}

\begin{proof}
The first relation follows from
\[
\frac{d^2}{dq^2}\bigl(q[x]_q+1\bigr)
=2\frac{d}{dq}[x]_q+q\frac{d^2}{dq^2}[x]_q.
\]
For the second relation, differentiate twice the second recurrence~\eqref{RecEq}.  
If \(f(q)=[x]_q\), then
\[
\left[-\frac1x\right]_q=-q^{-1}f(q)^{-1},
\]
and therefore
\[
\frac{d^2}{dq^2}\!\left(-q^{-1}f^{-1}\right)
=-2q^{-3}f^{-1}-2q^{-2}f'f^{-2}
+q^{-1}f''f^{-2}-2q^{-1}(f')^2f^{-3}.
\]
Evaluation at \(q=1\) gives~\eqref{eq:D2S} immediately.
\end{proof}

The relations~\eqref{eq:D2T}--\eqref{eq:D2S}, 
together with $D_2(0)=0$, determine $D_2$ on $\Q$, 
since the transformations $T:x\mapsto x+1$ and $S:x\mapsto-1/x$ 
generate the transitive action of the modular group.

\subsection{The generalized Dedekind sum of weight four}\label{ApoSec}

Let
$$
B_1(t)=t-\frac12,
\qquad\qquad
B_3(t)=t^3-\frac32t^2+\frac12t
$$
be low-degree Bernoulli polynomials.
We use their periodic versions
$$
\Bbar_1(t)=
\begin{cases}
\{t\}-\frac12,&t\notin\Z,\\
0,&t\in\Z,
\end{cases}
\qquad
\Bbar_3(t)=B_3(\{t\}),
$$
where $\{t\}$ is the fractional part of $t$.
For coprime integers $a,b$ with $b>0$, define
\begin{equation}
s_3(a,b)
:=
\sum_{n=1}^{b-1}
\Bbar_1\!\left(\frac nb\right)
\Bbar_3\!\left(\frac{an}{b}\right).
\label{eq:s3def}
\end{equation}
This is Apostol's generalized Dedekind sum of weight $4$; see~\cite{Apostol}.  
It is also the sum denoted by $s_{1,3}(a,b)$ in~\cite{Las}.  
The finite sum in Lasker's Theorem~1.2 is equal to $-s_3(a,b)$.

We will use the following three properties:
\begin{align}
s_3(a+b,b)&=s_3(a,b),
\label{eq:speriod}
\\
s_3(-a,b)&=-s_3(a,b),
\label{eq:sodd}
\end{align}
and the reciprocity law
\begin{equation}
 a b^3 s_3(a,b)+b a^3 s_3(b,a)
=
-\frac{a^4-5a^2b^2+b^4+3}{120},
\label{eq:reciprocity}
\end{equation}
for positive coprime $a,b$.  Formula~\eqref{eq:reciprocity} is the case $m=3$ of 
Apostol's reciprocity formula for generalized Dedekind sums~\cite{Apostol}.

Combining~\eqref{eq:sodd} and~\eqref{eq:reciprocity} also gives
\begin{equation}
 s_3(-b,a)
=
\frac{b^2}{a^2}s_3(a,b)
+
\frac{a^4-5a^2b^2+b^4+3}{120a^3b}.
\label{eq:Sinversion}
\end{equation}

\begin{thm}
[\cite{Las}, Theorem 1.2] 
\label{thm:Lasker}
Let $x=a/b$ be reduced, with $b>0$.  Then
\begin{equation}
\label{eq:D2formula}
D_2\!\left(\frac ab\right)
=
\frac{x^3}{3}-x^2+\frac{5x}{3}-1
+\frac{1-x}{b^2}
-20s_3(a,b).
\end{equation}
\end{thm}

\begin{proof}
Let the right-hand side of~\eqref{eq:D2formula} be denoted by $F(a/b)$.  
We prove that $F$ satisfies exactly the same recurrences as $D_2$.

First, $s_3(a+b,b)=s_3(a,b)$ and hence
\begin{align*}
F(x+1)-F(x)
&=
 x^2-x+1-\frac1{b^2}\\
&=2D_1(x),
\end{align*}
where we used~\eqref{FDEq}.  
Thus $F$ satisfies~\eqref{eq:D2T}.

It remains to check inversion. 
By the oddness relation~\eqref{eq:sodd}, the case \(a<0\) 
reduces to the same reciprocity identity with \(|a|\). 
 We therefore display the computation for \(a,b>0\).
Using~\eqref{eq:Sinversion}, with $x=a/b$, we obtain
\begin{align*}
F\!\left(-\frac1x\right)
={}&
P\!\left(-\frac1x\right)
+\frac{1+1/x}{a^2}
-20\frac{b^2}{a^2}s_3(a,b)\\
&\quad
-\frac{a^4-5a^2b^2+b^4+3}{6a^3b},
\end{align*}
where
$$
P(x):=\frac{x^3}{3}-x^2+\frac{5x}{3}-1.
$$
Substituting~\eqref{FDEq} and~\eqref{eq:D2formula} into the right-hand side of~\eqref{eq:D2S}, 
and simplifying, gives precisely this expression.  Therefore $F$ also satisfies~\eqref{eq:D2S}.

Finally,
$
F(0)=D_2(0),
$
since $0=0/1$ and $s_3(0,1)=0$.  The modular recurrences determine the function uniquely on $\Q$, hence $F=D_2$.
\end{proof}

\subsection{Expansion in $h=\log q$}

Write
\begin{equation}
[x]_{e^h}=x+A(x)h+B(x)h^2+O(h^3).
\label{eq:hExpansion}
\end{equation}
Since
\[
e^h-1=h+\frac12h^2+O(h^3),
\]
we have
\begin{equation}
A(x)=D_1(x),
\qquad
B(x)=\frac{D_2(x)+D_1(x)}2.
\label{eq:BhFromD}
\end{equation}
Combining~\eqref{FDEq} and~\eqref{eq:D2formula} yields the following formula.
\begin{align}
A(x)
&=
\frac12\left(x^2-x+1-\frac1{b^2}\right),
\label{eq:Aformula}\\
B(x)
&=
\frac{x^3}{6}-\frac{x^2}{4}+\frac{7x}{12}-\frac14
+\frac{1-2x}{4b^2}
-10s_3(a,b).
\label{eq:Bformula2}
\end{align}
These formulas will be useful below.

\section{Computing the first two coefficients in the $h = \log q$ expansion}\label{ComputSec}

In this section, we obtain several explicit formulas for the first two coefficients $\crr'$ and $\crr''$
in the $h = \log q$ expansion
and prove Theorems~\ref{FirstCFThm} and~\ref{SecondCFThm}, as well as Proposition~\ref{PProp}.
We assume throughout this section that the four points are finite and pairwise distinct.

\subsection{A coordinate formula for $\crr'$}

Differentiating the logarithm of the cross-ratio gives
\begin{align*}
\frac{\crr'(x_1,x_2;x_3,x_4)}{\crr(x_1,x_2;x_3,x_4)}
={}&
\frac{D_1(x_1)-D_1(x_3)}{x_1-x_3}
+
\frac{D_1(x_2)-D_1(x_4)}{x_2-x_4}
\\
&-
\frac{D_1(x_1)-D_1(x_4)}{x_1-x_4}
-
\frac{D_1(x_2)-D_1(x_3)}{x_2-x_3}.
\end{align*}
Write every rational number $x_i=a_i/b_i$, with $i=1,2,3,4$, in lowest terms
$$
x_i=\frac{a_i}{b_i},
\qquad
b_i>0,
\qquad
\gcd(a_i,b_i)=1
$$
and apply~\eqref{FDEq}.
The contributions of
$\frac{x^2-x+1}{2}$
cancel identically. 
Therefore only the denominator-dependent terms remain, and one obtains
\begin{equation}
 \label{CoorEq}
\begin{array}{rcl}
 \displaystyle\frac{\crr'(x_1,x_2;x_3,x_4)}{\crr(x_1,x_2;x_3,x_4)}
& =&
\displaystyle\frac12\bigg(
 \frac{b_3^{-2}-b_1^{-2}}{x_1-x_3}
 +\frac{b_4^{-2}-b_2^{-2}}{x_2-x_4}
\\[12pt]
& &\;\;\;
\displaystyle-\frac{b_4^{-2}-b_1^{-2}}{x_1-x_4}
 -\frac{b_3^{-2}-b_2^{-2}}{x_2-x_3}
 \bigg).
\end{array}
\end{equation}

Our next goal is to deduce~\eqref{CrT1} from~\eqref{CoorEq}.

\subsection{Proof of Theorem~\ref{FirstCFThm}}

Recall that we use the canonical primitive lift 
$
X_i=
\begin{pmatrix}
a_i\\
b_i
\end{pmatrix}.
$
Thus the Pl\"ucker determinants are explicitly given by
$$
\Delta_{ij}
=
\det(X_i,X_j)
=
a_i b_j-a_j b_i.
$$
In particular,
$$
x_i-x_j
=
\frac{a_i}{b_i}-\frac{a_j}{b_j}
=
\frac{\Delta_{ij}}{b_i b_j}.
$$
Every term in~\eqref{CoorEq} can thus be rewritten in the form
$$
\frac{\displaystyle b_j^{-2}-b_i^{-2}}
     {x_i-x_j}
=
\frac{\displaystyle \frac{b_i}{b_j}-\frac{b_j}{b_i}}{\Delta_{ij}}.
$$

Regrouping the first and the third terms in the right-hand side of~\eqref{CoorEq}, we have
$$
E_1:=\frac{\displaystyle \frac{b_1}{b_3}-\frac{b_3}{b_1}}{\Delta_{13}}-
\frac{\displaystyle \frac{b_1}{b_4}-\frac{b_4}{b_1}}{\Delta_{14}},
$$
and similarly for the second and the fourth terms, whose combination we denote by $E_2$.

\begin{lem}
\label{TeLem}
We have
\begin{equation}
\label{E1Eq}
E_1=\frac{
(\Delta_{14}-\Delta_{13})^2-\Delta_{34}^2}{
\Delta_{34}\Delta_{13}\Delta_{14}}
-
\frac{(b_3-b_4)^2}
     {b_3b_4\Delta_{34}},
\end{equation}
and similarly for $E_2$ obtained by $(1,3,4)\mapsto(2,4,3)$.
\end{lem}
\begin{proof}
Using the simple relation
$$
b_3\Delta_{14}-b_4\Delta_{13}=b_1\Delta_{34},
$$
which follows immediately from the definition of the determinants, we rewrite
\begin{align*}
E_1
&=
\frac{b_1}{b_3\Delta_{13}}-
\frac{b_3}{b_1\Delta_{13}}-
\frac{b_1}{b_4\Delta_{14}}+
\frac{b_4}{b_1\Delta_{14}}
\\
&=
\frac{b_1\bigl(b_4\Delta_{14}-b_3\Delta_{13}\bigr)}
     {b_3b_4\Delta_{13}\Delta_{14}}
+
\frac{b_4\Delta_{13}-b_3\Delta_{14}}
     {b_1\Delta_{13}\Delta_{14}}
\\
&=
\frac{b_1\bigl(b_4\Delta_{14}-b_3\Delta_{13}\bigr)}
     {b_3b_4\Delta_{13}\Delta_{14}}
-
\frac{\Delta_{34}}{\Delta_{13}\Delta_{14}}.
\end{align*}
Substituting
$$
b_1=\frac{b_3\Delta_{14}-b_4\Delta_{13}}{\Delta_{34}}
$$
in the first term gives
\begin{align*}
E_1
&=
\frac{
\bigl(b_3\Delta_{14}-b_4\Delta_{13}\bigr)
\bigl(b_4\Delta_{14}-b_3\Delta_{13}\bigr)}
{b_3b_4\Delta_{34}\Delta_{13}\Delta_{14}}
-
\frac{\Delta_{34}}{\Delta_{13}\Delta_{14}}.
\end{align*}
Now
\begin{align*}
\frac{
\bigl(b_3\Delta_{14}-b_4\Delta_{13}\bigr)
\bigl(b_4\Delta_{14}-b_3\Delta_{13}\bigr)}{b_3b_4}
&=
\Delta_{14}^2+\Delta_{13}^2
-
\frac{b_3^2+b_4^2}{b_3b_4}\,
\Delta_{13}\Delta_{14}
\\
&=
(\Delta_{14}-\Delta_{13})^2
-
\frac{(b_3-b_4)^2}{b_3b_4}\,
\Delta_{13}\Delta_{14}.
\end{align*}
Therefore
$$
E_1=
\frac{
(\Delta_{14}-\Delta_{13})^2-\Delta_{34}^2}
{\Delta_{34}\Delta_{13}\Delta_{14}}
-
\frac{(b_3-b_4)^2}{b_3b_4\Delta_{34}},
$$
which is~\eqref{E1Eq}.
\end{proof}

Adding the expressions for $E_1$ and $E_2$, the denominator-dependent terms cancel; 
since $\crr'/\crr=(E_1+E_2)/2$, this gives~\eqref{CrT1}.

\subsection{Proof of Proposition~\ref{PProp}}
It will suffice to give a pair of quadruples of points on~$\pP^1(\Q)$
for which all six Pl\"ucker determinants coincide, but $\crr''$
takes different values.

Consider the two ordered quadruples
$$
\mathcal X=
\left(\frac14,\frac34;\frac18,\frac38\right),
\qquad
\mathcal Y=
\left(\frac34,\frac54;\frac58,\frac78\right).
$$
Their normalized primitive lifts are respectively
$$
(1,4),\ (3,4),\ (1,8),\ (3,8),
\qquad
(3,4),\ (5,4),\ (5,8),\ (7,8).
$$
The second collection is obtained from the first by the determinant-one rational matrix
$$
A=
\begin{pmatrix}
1&\frac12\\
0&1
\end{pmatrix}
\in \mathrm{SL}(2,\Q).
$$
All resulting vectors remain primitive integral vectors.  Since
$\det A=1$, the six signed determinants coincide.  Explicitly, in both cases
one obtains
$$
(\Delta_{12},\Delta_{13},\Delta_{14},
  \Delta_{23},\Delta_{24},\Delta_{34})
=(-8,4,-4,20,12,-16).
$$

A direct computation from the continued-fraction definition of the
$q$-rationals gives
$$
\crr_q(\mathcal X)
=-\frac{q(q^2+q+1)}{q^5+q^4+q^2+q+1}
$$
and
$$
\crr_q(\mathcal Y)
=-\frac{(q^2-q+1)(q^2+q+1)(q^4-q^2+1)}{q^2(q^5+q^4+q^2+q+1)}.
$$
Consequently,
\begin{equation}
\label{eq:exact-difference}
\crr_q(\mathcal Y)-\crr_q(\mathcal X)
=
-\frac{(q-1)^2(q^2+q+1)(q^4+q^3+q^2+q+1)}
       {q^2(q^5+q^4+q^2+q+1)}.
\end{equation}
Substituting $q=e^h$ gives
\begin{align*}
\crr_{e^h}(\mathcal X)
&=-\frac35+\frac6{25}h+\frac{98}{125}h^2+O(h^3),\\
\crr_{e^h}(\mathcal Y)
&=-\frac35+\frac6{25}h-\frac{277}{125}h^2+O(h^3).
\end{align*}
The classical cross-ratios and the first coefficients therefore coincide,
whereas
\[
\crr''(\mathcal X)=\frac{196}{125},
\qquad
\crr''(\mathcal Y)=-\frac{554}{125}.
\]
This proves the claim.

\begin{cor}
The logarithmic second coefficient
\[
\frac12\left[
\frac{\crr''}{\crr}
-\left(\frac{\crr'}{\crr}\right)^2
\right]
\]
is likewise not determined by the six normalized Pl\"ucker determinants.
\end{cor}

\begin{rem}
The obstruction is arithmetic rather than merely computational.  The
classification of ordered triples under $\PGL(2,\Z)$ shows that the three
absolute pairwise determinants do not always determine the orbit: in general
one must add a residue class
$$
\kappa\pmod{\gcd(|\Delta_{12}|,|\Delta_{13}|,|\Delta_{23}|)},
$$
see Theorem~\ref{ResClassThm}.
For the first three points of the two quadruples above, the determinant data
are the same, while the corresponding residues are $1$ and $3$ modulo $4$.
Thus the example detects precisely the sort of integral lattice information
that ordinary Pl\"ucker determinants forget.
\end{rem}

\subsection{Proof of Theorem~\ref{SecondCFThm}}

For $i\ne j$, formula~\eqref{eq:hExpansion} gives
\begin{align*}
\log\!\left(
  \frac{[x_i]_{e^h}-[x_j]_{e^h}}{x_i-x_j}
\right)
={}&
\frac{A(x_i)-A(x_j)}{x_i-x_j}\,h
\\
&+
\left(
  \frac{B(x_i)-B(x_j)}{x_i-x_j}
  -\frac12
  \left(
    \frac{A(x_i)-A(x_j)}{x_i-x_j}
  \right)^2
\right)h^2
+O(h^3).
\end{align*}
Here and below, all logarithms are formal logarithms of power
series with constant term~$1$.

By the definition of the cross-ratio,
\begin{align*}
&\log\!\left(
  \frac{\crr_{e^h}(x_1,x_2;x_3,x_4)}
       {\crr(x_1,x_2;x_3,x_4)}
\right)
\\
&\qquad={}
\log\!\left(
  \frac{[x_1]_{e^h}-[x_3]_{e^h}}{x_1-x_3}
\right)
+
\log\!\left(
  \frac{[x_2]_{e^h}-[x_4]_{e^h}}{x_2-x_4}
\right)
\\
&\qquad\phantom{={}}-
\log\!\left(
  \frac{[x_1]_{e^h}-[x_4]_{e^h}}{x_1-x_4}
\right)
-
\log\!\left(
  \frac{[x_2]_{e^h}-[x_3]_{e^h}}{x_2-x_3}
\right).
\end{align*}
The coefficient of $h^2$ in the left-hand side is
\[
\frac12\left[
  \frac{\crr''}{\crr}
  -\left(\frac{\crr'}{\crr}\right)^2
\right].
\]

Fix $i\ne j$.
From~\eqref{eq:Aformula},
\begin{equation}
\label{PairA}
\frac{A(x_i)-A(x_j)}{x_i-x_j}
=
\frac12\left(
x_i+x_j-1-\frac{b_i^{-2}-b_j^{-2}}{x_i-x_j}
\right).
\end{equation}
Formula~\eqref{eq:Bformula2} gives
\begin{equation}
\label{PairB}
\begin{array}{rcl}
\displaystyle
\frac{B(x_i)-B(x_j)}{x_i-x_j}
&=&
\displaystyle
\frac{(x_i+x_j)^2}{8}
+\frac{(x_i-x_j)^2}{24}
-\frac{x_i+x_j}{4}
+\frac{7}{12}
\\[10pt]
&&\displaystyle+
\frac{(1-(x_i+x_j))(b_i^{-2}-b_j^{-2})}{4(x_i-x_j)}
-\frac{b_i^{-2}+b_j^{-2}}{4}
-10\frac{s_i-s_j}{x_i-x_j}.
\end{array}
\end{equation}
Combining~\eqref{PairA} and~\eqref{PairB}, the mixed term cancels, and we obtain
\begin{equation}
\begin{array}{l}
\displaystyle
\frac{B(x_i)-B(x_j)}{x_i-x_j}
-
\frac12
\left(
\frac{A(x_i)-A(x_j)}{x_i-x_j}
\right)^2
\\[10pt]
\displaystyle
\hskip1cm\qquad=
\frac{(x_i-x_j)^2}{24}
-
\frac18
\left(
\frac{b_i^{-2}-b_j^{-2}}{x_i-x_j}
\right)^2
-10\frac{s_i-s_j}{x_i-x_j}
+
\frac{11}{24}
-
\frac{b_i^{-2}+b_j^{-2}}{4}
\\[10pt]
\displaystyle
\hskip1cm\qquad=
\Phi_{ij}
+
\frac{11}{24}
-
\frac{b_i^{-2}+b_j^{-2}}{4}.
\end{array}
\label{PairSecondOrder}
\end{equation}

Finally, take the alternating sum of~\eqref{PairSecondOrder} for the pairs
$(x_1,x_3)$, $(x_2,x_4)$, $(x_1,x_4)$, and $(x_2,x_3)$.  The constant terms cancel, and
each $b_i^{-2}$ also occurs once with each sign.  What remains is
\[
\frac12\left[
\frac{\crr''}{\crr}
-\left(\frac{\crr'}{\crr}\right)^2
\right]
=
\Phi_{13}+\Phi_{24}-\Phi_{14}-\Phi_{23},
\]
which is~\eqref{CrT2}.

\section{Extensions and further questions}\label{SecQuest}

\subsection{$q$-deformation of rational friezes}

A very natural problem that can be tackled with the help of the $q$-cross-ratio
is the $q$-deformation of tame closed friezes with
entries in~$\Q$.  In the classical theory, frieze entries are expressed in
terms of Pl\"ucker determinants, and the second nontrivial row consists of
cross-ratios; see~\eqref{ClassPluFr} and~\eqref{CREntriesEq}.  The formulas of
Section~\ref{FrizSec} therefore suggest using $q$-rational points and
$q$-cross-ratios to construct a corresponding deformation of rational
friezes.

A basic issue is the choice of a projective representative of the
rational polygon. The polygon associated with a classical rational
frieze is defined only up to $\PGL(2,\Q)$, whereas the map
$x\mapsto[x]_q$ is only modularly equivariant. Thus a rational
projective change of coordinates need not preserve the resulting
$q$-cross-ratios. This is different from the choice of numerator and
denominator for a fixed rational vertex, which does not affect
its $q$-deformation.

The question is therefore whether a suitable normalization can make
this construction depend only on the rational frieze, or whether
additional arithmetic data must be included. One must also determine
whether the resulting array satisfies a natural $q$-diamond rule and
whether closure and periodicity are preserved. We refer to the recent
preprint~\cite{IMK} on degree shifts between $q$-friezes and $q$-Farey
labelings, which may be related to this problem.

\subsection{Left and mixed $q$-cross-ratios}
\label{subsec:mixed-cross-ratios}

The definition~\eqref{Crq} uses the original, or right, $q$-rational.
However, it is also possible to use left $q$-rationals, see Section~\ref{LeftSec}.

Following~\cite{BBL,Jou,JPT} and for convenience, let us use the notation $[x]^{\sharp}_q=[x]_q$.
For a pattern
\[
 \boldsymbol\varepsilon=(\varepsilon_1,\varepsilon_2,
 \varepsilon_3,\varepsilon_4),
 \qquad \varepsilon_i\in\{\sharp,\flat\},
\]
define the corresponding mixed $q$-cross-ratio by
\begin{equation}
\label{MixedCrDefinition}
 \crr_q^{\boldsymbol\varepsilon}(x_1,x_2;x_3,x_4)
 :=
 \crr\bigl(
 [x_1]^{\varepsilon_1}_q,[x_2]^{\varepsilon_2}_q;
 [x_3]^{\varepsilon_3}_q,[x_4]^{\varepsilon_4}_q
 \bigr).
\end{equation}
In particular,
$\crr_q^{\sharp\sharp\sharp\sharp}=\crr_q$.
Let $\overline{\boldsymbol\varepsilon}$ be the complementary pattern,
obtained by interchanging $\sharp$ and $\flat$ in every position.

\begin{prop}
\label{prop:mixed-duality}
For every pattern $\boldsymbol\varepsilon\in\{\sharp,\flat\}^4$,
\begin{equation}
\label{MixedDuality}
 \crr_q^{\boldsymbol\varepsilon}
 =
 \crr_{q^{-1}}^{\overline{\boldsymbol\varepsilon}}.
\end{equation}
In particular,
\[
 \crr_q^{\flat\flat\flat\flat}=\crr_{q^{-1}}.
\]
\end{prop}

\begin{proof}
Apply the same M\"obius transformation~\eqref{TransitionMap} to all four
entries in~\eqref{MixedCrDefinition}.  The ordinary cross-ratio is
unchanged, while~\eqref{LeftRightTransition} complements every right--left
label and replaces $q$ by $q^{-1}$.
\end{proof}

There are $2^4=16$ patterns.  Complementation has no fixed pattern, so
these patterns form eight pairs under~\eqref{MixedDuality}.  For instance, one pair is
$
 \sharp\sharp\sharp\sharp
 \quad\longleftrightarrow\quad
 \flat\flat\flat\flat.
$
The remaining fourteen patterns form seven mixed duality classes.  A
convenient list of representatives is
\begin{equation}
\label{SevenMixedPatterns}
\begin{gathered}
 \flat\sharp\sharp\sharp,\quad
 \sharp\flat\sharp\sharp,\quad
 \sharp\sharp\flat\sharp,\quad
 \sharp\sharp\sharp\flat,\\
 \flat\flat\sharp\sharp,\quad
 \flat\sharp\flat\sharp,\quad
 \flat\sharp\sharp\flat.
\end{gathered}
\end{equation}
Here ``seven'' counts pattern classes; it is not a statement of algebraic
independence.  If one also identifies patterns under the Klein four-group,
the seven mixed classes reduce to four types, represented by
\[
 \flat\sharp\sharp\sharp,
 \qquad
 \flat\flat\sharp\sharp,
 \qquad
 \flat\sharp\flat\sharp,
 \qquad
 \flat\sharp\sharp\flat.
\]
Every mixed cross-ratio is invariant under the simultaneous
$\PSL(2,\Z)$-action.  Under the $q$-preserving extension of the action to
$\PGL(2,\Z)$, a determinant-$(-1)$ element exchanges every $\sharp$ with
$\flat$; see~\cite{Jou,JPT}.  Thus complementary pairs of patterns are the
natural objects for the full modular group.

\begin{rem}
The mixed cross-ratios genuinely contain more information than the pure
ones.  For the family
\[
 X_n=\left(-n,-\frac1n;1,-1\right),
 \qquad n\geq2,
\]
Proposition~\ref{prop:infinite-klein-fiber} gives
$\crr_q(X_n)=-q$ for every~$n$.  On the other hand,
$[-1]^{\flat}_q=-q^{-2}$, and direct substitution gives
\begin{equation}
\label{MixedXn}
 \crr_q^{\sharp\sharp\sharp\flat}(X_n)
 =
 \frac{q\bigl(q-[n]_q\bigr)}
      {(q^2-q+1)[n]_q-1}.
\end{equation}
This expression determines $[n]_q$ as a rational function of $q$ and hence~$n$; indeed, if its value is
$C$, then
\[
 [n]_q=\frac{C+q^2}{(q^2-q+1)C+q}.
\]
Thus mixed cross-ratios separate the infinite family that is invisible to
the two pure versions.
\end{rem}

Nevertheless, even the complete mixed family is not a complete modular
invariant.

\begin{ex}
\label{ex:mixed-nonseparation}
Consider the ordered quadruples
$$
 \mathbf x=
 \left(\infty,\frac13,\frac43,\frac{10}{3}\right),
 \qquad
 \mathbf y=
 \left(\infty,\frac23,\frac53,\frac{11}{3}\right).
$$
Then it can be checked by a direct computation that
$$
 \crr_q^{\boldsymbol\varepsilon}(\mathbf x)
 =
 \crr_q^{\boldsymbol\varepsilon}(\mathbf y)
 \qquad
 \text{for every }\boldsymbol\varepsilon\in\{\sharp,\flat\}^4.
$$
Indeed, the same affine map
\[
 H_q(z)=\frac{(1+q^2+q^3)z+q}{1+q+q^3}
\]
sends $[x_i]^\sharp_q$ to $[y_i]^\sharp_q$ and $[x_i]^\flat_q$ to $[y_i]^\flat_q$, for every $i$.
However, $\mathbf x$ and $\mathbf y$ are not equivalent under the joint
action of $\PGL(2,\Z)$ and~$S_4$.
Indeed, the unique projective transformation sending $\mathbf x$ to $\mathbf y$ and preserving
their Pl\"ucker invariants is represented by the matrix
\[
 \begin{pmatrix}3&1\\0&3\end{pmatrix},
\]
whose determinant is~$9$ rather than~$\pm1$.  It does not belong to
$\PGL(2,\Z)$.
\end{ex}

The mixed versions are strictly stronger than the pure $q$-cross-ratio, but still fail to distinguish some modular orbits. 
This leads naturally to the symmetric classification problem considered next.

\subsection{Toward a symmetric classification}

Proposition~\ref{prop:infinite-klein-fiber} shows that the original
$q$-cross-ratio is far from being a complete invariant: a single fiber
contains infinitely many equivalence classes even after quotienting by the
Klein four-group.  Since the equality in~\eqref{eq:infinite-klein-fiber} is
an identity of rational functions in~$q$, the entire sequence
\[
 \crr,\crr',\crr'',\ldots
\]
also fails to separate these classes.  The mixed cross-ratios introduced in
Section~\ref{subsec:mixed-cross-ratios} contain strictly more information,
but Example~\ref{ex:mixed-nonseparation} shows that even their complete
labelled collection does not separate modular orbits.  This does not
contradict Theorem~\ref{InvThm}, which concerns algebraic independence of
the invariant functions as the quadruple varies.

Appendix~\ref{PairTripleApp} gives an elementary, although asymmetric,
complete classification of ordered quadruples modulo $\PGL(2,\Z)$.  One
first puts the ordered triple $(x_1,x_2,x_3)$ into the unique normal form
\[
 \left(\infty,\frac{u}{n},\frac{r}{m}\right)
\]
of Theorem~\ref{NFTThm}.  If $z$ is the image of the fourth point, then
formula~\eqref{eq:reconstruct-classical-fourth} reconstructs $z$ from the
classical cross-ratio
\[
 \lambda=\crr\left(\infty,\frac{u}{n};\frac{r}{m},z\right).
\]
Thus the normal-form data $(n,m,u,r)$ of the first triple together with the
classical cross-ratio form a complete invariant of ordered quadruples.  For
the Klein quotient, one obtains a complete but indirect invariant by taking
the corresponding $V_4$-orbit of this data.

This leads to a more natural classification problem.

\medskip
\noindent\textbf{Problem.}
Find a direct complete invariant of rational quadruples for the joint action
of $\PGL(2,\Z)$ and~$V_4$ that does not privilege an ordered triple.  In
particular, determine what finite residue data must be added to the
$V_4$-orbit of the six absolute normalized Pl\"ucker determinants in order
to obtain a complete invariant.
\medskip

The finiteness statement in Proposition~\ref{prop:fixed-plucker-finite}
shows that, once the six labelled absolute Pl\"ucker determinants are fixed,
only finitely many $\PGL(2,\Z)$-orbits remain.  The problem is therefore not
to find further continuous parameters, but to describe the missing finite
arithmetic information in a symmetric and effective way.  The
$q$-cross-ratio may still encode useful parts of this information, but it
cannot by itself serve as a classifier.

\subsection{Formal extension to real points}

For $x\in\R$, the $q$-deformation $[x]_q$ is defined as a formal Laurent
series at $q=0$; for irrational $x$, see~\cite{MGOexp}.
A version of complex $q$-deformed numbers was suggested by Etingof~\cite{Eti}.
Formula~\eqref{Crq} therefore defines a formal $q$-cross-ratio for pairwise distinct real, or complex points. 
For complex points in the upper half-plane we use Etingof's analytic construction, 
discussed in the next subsection. 
The $\PSL(2,\Z)$-equivariance of $q$-real and complex numbers implies the
invariance of this formal cross-ratio.

Two points require additional care in this setting.
First, a formal expansion at $q=0$ need not admit an analytic or formal
expansion at $q=1$, so the coefficients $\crr',\crr'',\ldots$ considered in
this paper are not automatically defined
(see however \cite{Eti,EO} for analytic properties of $q$-reals).
It would be interesting to
identify natural classes of real points for which the $q$-cross-ratio admits
continuation to a neighborhood of $q=1$. 
Second, the determinant-$(-1)$
covariance law uses the substitution $q\mapsto q^{-1}$, which is generally
not defined on a one-sided Laurent series.
Jouteur's $q$-preserving $\PGL(2,\Z)$ action~\cite{Jou} may provide a way around this obstruction; 
its effect on the present cross-ratio remains to be worked out.
It would be interesting to determine which parts of the
$\PGL(2,\Z)$ covariance survive there.

\subsection{Complex arguments and the Eisenstein series $E_4$}
\label{subsec:complex-eisenstein}

Etingof's complex $q$-numbers \cite[Section~9]{Eti}
provide an analytic counterpart of the construction considered in this
paper. 

For four distinct points $\tau_1,\ldots,\tau_4$ in the upper
half-plane $\mathbb H$, put
\[
 C_h=\operatorname{cr}\bigl([\tau_1]_{e^h},[\tau_2]_{e^h};
 [\tau_3]_{e^h},[\tau_4]_{e^h}\bigr).
\]
We work with the germ at $h=0$. 
We also put $C:=C_0=\crr(\tau_1,\tau_2;\tau_3,\tau_4)$,
where $[\tau]_{e^0}=\tau$.
Define
\begin{equation}\label{eq:complex-polarized-eisenstein}
 \mathcal E_4(\tau_1,\tau_2,\tau_3,\tau_4)
 =\frac12\sum_{\substack{(c,d)\in\mathbb Z^2\\\gcd(c,d)=1}}
 \frac{1}{\prod_{j=1}^4(c\tau_j+d)}.
\end{equation}
This series converges absolutely and locally uniformly on $\mathbb H^4$.
It is symmetric, and on the diagonal equals the normalized Eisenstein
series:
\[
 \mathcal E_4(\tau,\tau,\tau,\tau)=E_4(\tau),
 \qquad E_4(\tau)=1+240e^{2\pi i\tau}+\cdots.
\]

\begin{prop}
The germ $C_h$ is even in $h$, and
\begin{equation}
\label{eq:CCR}
 C_h=C\left(1-\frac{h^2}{12}
 (\tau_1-\tau_2)(\tau_3-\tau_4)
 \mathcal E_4(\tau_1,\tau_2,\tau_3,\tau_4)+O(h^4)\right).
\end{equation}
In particular, the first derivative with respect to $q$ at $q=1$ vanishes.
\end{prop}

\begin{proof}
Let $f_h(\tau)=[\tau]_{e^h}$, and write
$\mathcal S_\tau(f)=f'''/f'-\tfrac32(f''/f')^2$ for the Schwarzian
derivative with respect to $\tau$.
Etingof's~\cite{Eti} equation~(9.3), with $J=j/1728$, and the Schwarzian chain rule
give
\[
 \mathcal S_\tau(f_h)
 =\frac{h^2}{8\pi^2}\frac{J'(\tau)^2}{J(\tau)(J(\tau)-1)}
 =-\frac{h^2}{2}E_4(\tau).
\]
Thus $f_h$ and $f_{-h}$ have the same Schwarzian and differ by a
M\"obius transformation, proving $C_h=C_{-h}$.

The linear coefficient of $f_h$ has zero third derivative, so it can be
removed by M\"obius postcomposition. The normalized family then has the
form $\tau+h^2v(\tau)+O(h^3)$, where $v'''=-E_4/2$.
Expanding the cross-ratio gives
\[
 \frac{C_h}{C}
 =1+h^2(\tau_1-\tau_2)(\tau_3-\tau_4)
 \sum_{i=1}^4\frac{v(\tau_i)}{\prod_{j\ne i}(\tau_i-\tau_j)}
 +O(h^3).
\]
Termwise triple integration of the primitive lattice expansion of $E_4$
yields
\[
 \sum_{i=1}^4\frac{v(\tau_i)}{\prod_{j\ne i}(\tau_i-\tau_j)}
 =-\frac1{12}\mathcal E_4(\tau_1,\tau_2,\tau_3,\tau_4),
\]
implying~\eqref{eq:CCR}.
Indeed, for $c\ne0$ a triple primitive of
$-\tfrac12(c\tau+d)^{-4}$ is $1/(12c^3(c\tau+d))$, whose third divided
difference is $-1/(12\prod_j(c\tau_j+d))$; the $c=0$ terms give the same
identity. Integration constants are quadratic polynomials and disappear
from the divided difference. Normal convergence justifies this computation
after fixing the three integration constants at a basepoint.
Evenness improves the remainder to $O(h^4)$.
\end{proof}

This analytic construction is also related to mixed cross-ratios.
Fix $0<q<1$ and choose $M>0$ sufficiently large as
in~\cite[Section~9.2]{Eti}.
For $x\in\Q$, choose $g\in\PSL(2,\Z)$ with $g(\infty)=x$.
A sequence $\tau_n\in\mathbb H$ approaching $x$ is $M$-bounded if
\[
\operatorname{Im}(g^{-1}\tau_n)\ge M
\qquad\text{for all }n.
\]
It approaches from the right (respectively, left) if
$\operatorname{Re}(g^{-1}\tau_n)\to-\infty$
(respectively, $+\infty$); see~\cite[Definition~9.8]{Eti}.
Under these hypotheses,~\cite[Proposition~9.10]{Eti} gives
\[
[\tau_n]_q\longrightarrow
\begin{cases}
[x]_q, & \text{for a right approach},\\
[x]^\flat_q, & \text{for a left approach},
\end{cases}
\]
see also~\cite[Remark 3.2]{MGOexp}.
Taking such limits independently in the four arguments produces
the corresponding mixed cross-ratios, whenever the limiting
denominator is nonzero.

The expansion above concerns fixed interior points; it does not
justify interchanging differentiation at $q=1$ with passage to the
rational boundary. In particular, its vanishing first variation does
not contradict the generally nonzero first variation for rational
arguments.

\appendix
\section{Pairs, triples, and ordered quadruples of rational points}
\label{PairTripleApp}

We describe here $\PGL(2,\Z)$-orbits of ordered pairs, triples,
and quadruples of distinct points of~$\pP^1(\Q)$.
Pairs of points were classified by Jones, Singerman, and Wicks~\cite{JSW};
we recall the corresponding classification.
We then give a classification of triples
and, as its immediate consequence,
of ordered quadruples.

\subsection{Pairs}

Jones, Singerman, and Wicks~\cite{JSW} describe the
$\PSL(2,\Z)$-orbits of ordered pairs: every non-diagonal orbit contains
$(\infty,u/n)$, where $u$ is a unit modulo $n$.
Passing to $\PGL(2,\Z)$ adds the reflection $x\mapsto-x$ and gives the
following form of their result.

\begin{prop}
Every ordered pair of distinct points of $\pP^1(\Q)$ is
$\PGL(2,\Z)$-equivalent to
$$
\left(\infty,\frac{u}{n}\right),
\qquad
n\geq1,\quad 0\leq u<n,\quad \gcd(u,n)=1.
$$
Two such pairs, with parameters $(n,u)$ and $(n',u')$, are equivalent
if and only if
$$
n=n',\qquad u'\equiv \pm u\pmod n.
$$
\end{prop}

\begin{proof}
Send the first point to $\infty$ and write the second one as $u/n$ in
lowest terms, with $n>0$.  The stabilizer of $\infty$ in
$\PGL(2,\Z)$ consists of the transformations
$x\mapsto\varepsilon x+k$, where $\varepsilon=\pm1$ and $k\in\Z$.
Thus $n$ is unchanged and $u$ is determined modulo $n$ up to sign.
\end{proof}

Note that the absolute Pl\"ucker invariant of the pair of points $\left(\infty,\frac{u}{n}\right)$
equals~$n$.
It is constant on each $\PGL(2,\Z)$-orbit.

\begin{rem}
For an unordered pair one must also identify $u$ with $u^{-1}$ modulo
$n$: interchanging the two points sends the $\PSL(2,\Z)$-parameter
$u$ to $-u^{-1}$, as observed in~\cite{JSW}, and the minus sign is
already invisible for the $\PGL(2,\Z)$-action.
\end{rem}

\subsection{Triples: classification by normal form}

For an ordered triple $(x_1,x_2,x_3)$ of distinct points of $\pP^1(\Q)$,
the three absolute Pl\"ucker invariants
\begin{equation}
\label{PiT}
\left|\Delta_{12}\right|,
\quad\left|\Delta_{13}\right|,
\quad\left|\Delta_{23}\right|
\end{equation}
are integer invariants of the $\PGL(2,\Z)$-action.
However, these invariants 
do not, in general, classify such triples.  
The missing datum is one residue modulo their common divisor.

\begin{thm}
\label{NFTThm}
Every ordered triple $(x_1,x_2,x_3)$ of distinct points of
$\pP^1(\Q)$ is $\PGL(2,\Z)$-equivalent to a unique triple
\begin{equation}
\label{NFormEq}
\left(\infty,\frac{u}{n},\frac{r}{m}\right)
\end{equation}
satisfying
$$
n,m>0,\qquad 0\leq u<n,\qquad
\gcd(u,n)=\gcd(r,m)=1,
$$
and
$$
um-rn>0.
$$
\end{thm}

\begin{proof}
By transitivity of the $\PGL(2,\Z)$-action, 
there exists an
element of $\PGL(2,\Z)$ that sends $x_1$ to $\infty$.  
Write the other two
points in lowest terms as $a/n$ and $b/m$, with $n,m>0$, and consider
$$
\Delta_{23}=am-bn\neq0.
$$
Applying $x\mapsto-x$ if necessary makes
$\Delta_{23}>0$.
Finally, a unique integral translation (which belongs to the stabilizer of $\infty$) then puts the numerator
of the second point in the range $0\leq u<n$.  
This gives the stated normal form.

Any transformation fixing $\infty$ has the form
$x\mapsto\varepsilon x+k$.  
Positivity of $\Delta_{23}$ forces
$\varepsilon=1$, and the condition $0\leq u<n$ forces $k=0$.
Hence the normal form is unique.
\end{proof}

\begin{rem}
Note that the Pl\"ucker determinants of the normal form~\eqref{NFormEq}
are
$$
\Delta_{12}=n,
\qquad
\Delta_{13}=m,
\qquad
\Delta_{23}=um-rn.
$$
It follows that there are finitely many $\PGL(2,\Z)$-orbits with fixed invariants~\eqref{PiT}.
\end{rem}

\subsection{Triples: classification by residue}
The normal form admits the following equivalent description in terms of a residue class.
Fix the three Pl\"ucker invariants~\eqref{PiT}. 
The corresponding $\PGL(2,\Z)$-orbits
are classified by one residue class modulo
the greatest common divisor of the three invariants.

Let us put
$$
g:=\gcd(|\Delta_{12}|,|\Delta_{13}|,|\Delta_{23}|).
$$
Recall that $\Delta_{12}=n,
\Delta_{13}=m,
\Delta_{23}=um-rn$. 
Let
$$
\Delta_{12}=gN,
\qquad \Delta_{13}=gM,
\qquad \Delta_{23}=gD.
$$
From the primitivity conditions one obtains
$$
\gcd(n,\Delta_{23})=\gcd(n,m)=\gcd(m,\Delta_{23})=g.
$$
Hence $N,M,D$ are pairwise coprime.  

Let $u_0$ be the unique integer
such that
$$
0\leq u_0<N,
\qquad
u_0M\equiv D\pmod N.
$$
For the canonical normal form~\eqref{NFormEq} there is a unique residue
$\kappa\pmod g$ such that
$$
u=u_0+N\kappa.
$$
Note also that the representative $\kappa$ of the residue class can be chosen 
such that $0\leq\kappa<g$.
Put
$$
r=\frac{uM-D}{N}.
$$
The residue $\kappa$ is called admissible if
$\gcd(u,g)=\gcd(r,g)=1$.

The following statement provides a classification by one residue.

\begin{thm}
\label{ResClassThm}
The $\PGL(2,\Z)$-orbit is completely determined by the positive
 Pl\"ucker data
$$
(\Delta_{12},\Delta_{13},\Delta_{23})=(n,m,um-rn)
$$
of the canonical normal form, together with an admissible residue class $\kappa\;(\!\!\!\mod g)$.
Conversely, every admissible residue class determines a unique orbit.
\end{thm}

\begin{proof}
From $\Delta_{23}=um-rn$ and the primitivity conditions one obtains
$$
\gcd(n,\Delta_{23})=\gcd(n,m)=\gcd(m,\Delta_{23})=g.
$$
Hence $N,M,D$ are pairwise coprime.  
Dividing the relation $\Delta_{23}=um-rn$ by $g$ gives
$$
D=uM-rN,
$$
so $uM\equiv D\pmod N$.  
Since $M$ is invertible modulo $N$, this
fixes $u$ modulo $N$, and the range $0\leq u<gN$ gives the unique
expression $u=u_0+N\kappa$.  
The displayed formula then determines~$r$.  
The congruence makes $r$ integral; moreover it
implies $\gcd(u,N)=\gcd(r,M)=1$, so the two conditions $\gcd(u,g)=\gcd(r,g)=1$ 
are precisely the remaining primitivity conditions.
The canonical normal form now proves completeness.
\end{proof}

\begin{rem}
Conversely, we have the following way to reconstruct the normal form.
Let \(0\leq\kappa<g\) be the representative of an admissible residue class.
It reconstructs the triple by
$$
u=u_0+N\kappa,
\qquad
r=\frac{uM-D}{N},
$$
where admissibility means
$\gcd(u,g)=\gcd(r,g)=1$.
\end{rem}

If $g=1$, the residue disappears and the three integers
$\left|\Delta_{12}\right|,
\left|\Delta_{13}\right|,
\left|\Delta_{23}\right|$ alone classify the ordered triple.

\subsection{Ordered quadruples}
\label{subsec:quadruple-normal-form}

The classification of ordered triples immediately gives an asymmetric but
complete classification of ordered quadruples.

\begin{cor}
\label{cor:quadruple-normal-form}
Every ordered quadruple $(x_1,x_2,x_3,x_4)$ of pairwise distinct points of
$\pP^1(\Q)$ is $\PGL(2,\Z)$-equivalent to a unique quadruple
\begin{equation}\label{eq:quadruple-normal-form}
 \left(\infty,\frac{u}{n},\frac{r}{m},z\right),
\end{equation}
where
$$
 n,m>0,\qquad 0\leq u<n,\qquad
 \gcd(u,n)=\gcd(r,m)=1,\qquad um-rn>0,
$$
and $z$ is a (finite) rational such that
$$
 z\in\Q\setminus
 \left\{\frac{u}{n},\frac{r}{m}\right\}.
$$
\end{cor}

\begin{proof}
Apply Theorem~\ref{NFTThm} to the ordered triple $(x_1,x_2,x_3)$.  This produces a
unique normal form
$$
 \left(\infty,\frac{u}{n},\frac{r}{m}\right).
$$
Apply the same element of $\PGL(2,\Z)$ to $x_4$, and denote its image by
$z$.  Since $x_4\neq x_1$, the point $z$ is finite.  Since the four original
points are pairwise distinct, $z$ is different from $u/n$ and $r/m$.

It remains only to prove uniqueness of $z$.  Any two modular transformations
which put the first three points into the same ordered normal form differ by
a projective transformation fixing three distinct points.  Such a projective
transformation is the identity.  Hence the image $z$ of the fourth point is
unique.
\end{proof}

\begin{rem}
\label{rem:quadruple-reconstruction}
The fourth coordinate in~\eqref{eq:quadruple-normal-form} can be reconstructed
from the normal-form triple and the classical cross-ratio.  Indeed, if
\[
 \lambda=
 \crr\left(\infty,\frac{u}{n};\frac{r}{m},z\right),
\]
then
\begin{equation}
\label{eq:reconstruct-classical-fourth}
 \lambda=
 \frac{\frac{u}{n}-z}{\frac{u}{n}-\frac{r}{m}},
 \qquad
 z=\frac{u}{n}-\lambda
 \left(\frac{u}{n}-\frac{r}{m}\right).
\end{equation}
Thus the normal-form data of the first triple together with the classical
cross-ratio form a complete invariant of ordered quadruples modulo
$\PGL(2,\Z)$.
\end{rem}

\begin{prop}
\label{prop:fixed-plucker-finite}
There are only
finitely many $\PGL(2,\Z)$-orbits of ordered quadruples of pairwise distinct
rational points with fixed absolute values $|\Delta_{ij}|$, where $1\leq i<j\leq4$.  
More precisely, if
\[
 g=\gcd(|\Delta_{12}|,|\Delta_{13}|,|\Delta_{23}|),
\]
then the number of such orbits is at most $2g$. 
\end{prop}

\begin{proof}
By Theorem~\ref{NFTThm} and its finiteness remark, the three values
$|\Delta_{12}|,|\Delta_{13}|,|\Delta_{23}|$ allow only finitely many orbits of the ordered first
triple.  Choose the unique normal form for each of these finitely many triple
orbits.  The classical cross-ratio satisfies~\eqref{ClassEq}.
Thus the six absolute determinants allow at most two values of the classical
cross-ratio.  For each fixed normal-form triple and each such value, the
fourth point is uniquely reconstructed.
By Theorem~\ref{ResClassThm}, the first three absolute determinants give at most $g$
triple orbits, one for each admissible residue $\kappa\pmod g$.  Therefore
there are at most $2g$ quadruple orbits.  If all six signed determinants are
fixed, the classical cross-ratio itself, rather than only its absolute value,
is fixed, and the bound improves to~$g$.
\end{proof}

\begin{rem}
Proposition~\ref{prop:fixed-plucker-finite} explains precisely how
Proposition~\ref{PProp} should be interpreted.  
The six normalized Pl\"ucker
determinants need not determine one orbit, and they need not determine
$\crr''$; nevertheless, they leave only finitely many possible orbits. 
\end{rem}

\bigskip

\noindent{\bf Acknowledgments}.
This project took off during the conference
{\it Billiards and Stars: Geometry and Dynamics},
held in honor of the seventieth birthdays of Richard Montgomery and Sergei Tabachnikov.
I am grateful to the
Centro de Investigaci\'on en Matem\'aticas in Guanajuato, Mexico, for its hospitality.
I am indebted to 
Pavel Etingof, Sophie Morier-Genoud, and Sergei Tabachnikov for enlightening discussions.

ChatGPT (OpenAI) was used for exploratory computations, literature searches, and language editing. 
All mathematical arguments and computations were independently checked and rewritten.
I assume full responsibility for the content.


\bigskip

\noindent
{Valentin Ovsienko,
Centre National de la Recherche Scientifique,
Laboratoire de Math\'e\-ma\-tiques,
Universit\'e de Reims Champagne Ardenne,
Moulin de la Housse - BP 1039,
51687 Reims cedex 2,
France},\\
{Email:  valentin.ovsienko@univ-reims.fr}

\end{document}